\documentclass[12pt]{article}
\usepackage{amsfonts,amssymb,amsmath,amsthm,amsrefs}
\usepackage[mathscr]{euscript}
\usepackage [latin1]{inputenc}
\usepackage{color}
\newcommand{\R}{\mathbb{R}}

\numberwithin{equation}{section}

\newtheorem{thm}{Theorem}[section]
\newtheorem{cor}[thm]{Corollary}
\newtheorem{lem}[thm]{Lemma}

\newtheorem{rem}[thm]{Remark}

\allowdisplaybreaks

\usepackage{fancyhdr}
\begin{document}
\title{Global gradient estimates \\ for nonlinear parabolic operators}

\author{Serena Dipierro${}^{(1)}$
\and
Zu Gao${}^{(2)}$
\and
Enrico Valdinoci${}^{(1)}$
}

\maketitle

{
\scriptsize \begin{center} (1) -- Department of Mathematics and Statistics,
University of Western Australia\\ 35 Stirling Highway, WA6009 Crawley (Australia)\\
\end{center}
\scriptsize \begin{center} (2) --
Department of Mathematics, School of Science,
Wuhan University of Technology\\
122 Luoshi Road, 430070 Hubei, Wuhan (China) \end{center}

\begin{center}
{\tt serena.dipierro@uwa.edu.au},
{\tt gaozu7@whut.edu.cn},
{\tt enrico.valdinoci@uwa.edu.au}
\end{center}
}
\bigskip\bigskip

\par
\noindent
\centerline{\today}
\begin{abstract}\noindent
We consider a parabolic equation driven by a nonlinear
diffusive operator and we obtain a gradient
estimate in the domain where the equation takes place.

This estimate depends on the structural constants
of the equation, on the geometry of the ambient space
and on the initial and boundary data.

As a byproduct, one easily obtains a universal interior
estimate, not depending on the parabolic data.

The setting taken into account includes sourcing terms
and general diffusion coefficients. The results are new,
to the best of our knowledge, even in the Euclidean setting,
though we treat here also
the case of a complete Riemannian manifold.
\\
\\
\\
\noindent \textbf{Keywords}: Parabolic equations on Riemannian manifolds, maximum principle, global gradient estimates.
\\
\\
\textbf{MSC 2010}: 35B09, 35B50, 35K05, 35R01.
\end{abstract}

\section{Introduction}

The heat equation was introduced almost
two centuries ago by
Joseph Fourier~\cite{MR2856180}.
In spite of its classical flavor, the investigation
of the main properties of the solution
is still an active field of research,
and several important
gradient estimates have been obtained in modern literature.
Also, given its importance in geometric evolution problems,
some of these results have been framed into the framework
of Riemannian manifolds.
Among the several results on this topic,
we recall the following universal bound for compact manifolds:

\begin{thm}[Theorem 1.1 in~\cite{MR1230276}]
Let~$\mathscr{M}$
be a compact Riemannian manifold
with
$\mathrm{Ric}(\mathscr{M})\geq-k$, for some~$k\ge0$.
Let~$ u=u(x,t)$ be a positive solution of~$u_t=\Delta u$ in~$\mathscr{M}\times (0,+\infty)$.
Assume that~$u\le M$ for some~$M>0$.

Then, for each~$(x,t)\in\mathscr{M}\times(0,+\infty)$,
\begin{equation} \label{HAM}
\frac{t |\nabla u(x,t)|^2}{u^2(x,t)}\le (1+2kt)\,\ln \frac{M}{u(x,t)}.\end{equation}
\end{thm}

This type of result is certainly striking
and also somewhat surprising,
since typically parabolic estimates aim at controlling positive solutions at a given time by values at a later time, in view of the diffusive character of the equation
(see e.g. the classical Harnack Inequality on page~89 of \cite{MR0181836}), while
Richard Hamilton's estimate in~\eqref{HAM}
is a pointwise estimate in space-time.

As a matter of fact, an estimate of this type
cannot hold in non-compact manifolds,
as the simple case of the fundamental solution in~$\R^n$
shows: namely, taking~$u(x,t):=
\frac{1}{ (4\pi t)^{\frac{n}2} }
\exp\left(-\frac{|x|^2}{4t}\right)$, one sees that
$$ \frac{t\,|\nabla u|^2}{u^2}=\frac{|x|^2}{4t}
$$
which does not permit a global bound as in~\eqref{HAM}.

With respect to this, several gradient estimates
have been obtained in non-compact manifolds
by considering ``interior estimates'' in space and time.
More specifically, if~$x_0\in\mathscr{M}$
and~$R>0$, one denotes the geodesic ball
of radius $R$ centered at~$x_0$ by $B(x_0,R)$.
Also, given $t_0\in\R$ and~$T>0$, we let
\begin{equation}\label{QRT}
Q_{R,T}:=B(x_0,R)\times [t_0-T,t_0].\end{equation}
In this setting, a fruitful topic of investigation consists
in obtaining local gradient estimates in~$Q_{R/2,T/2}$,
see especially the work~\cite{MR834612}
by Peter Li and Shing-Tung Yau in which the maximum principle and suitable cut-off functions
have been used to obtain a parabolic Harnack inequality on complete Riemannian manifolds. In this setting,
we recall also a celebrated result by
Philippe Souplet and Qi S. Zhang:

\begin{thm}[Theorem~1.1 in~\cite{MR2285258}]\label{HEAT1}
Let~$\mathscr{M}$
be a complete Riemannian manifold
with $\mathrm{Ric}(\mathscr{M})\geq-k$, for some~$k\ge0$.
Let~$ u=u(x,t)$ be a positive solution of~$u_t=\Delta u$ in~$Q_{R,T}$.
Assume that~$u\le M$ for some~$M>0$.

Then, for each~$(x,t)\in Q_{R/2,T/2}$,
\begin{equation*}
\frac{|\nabla u(x,t)|}{u(x,t)}\le
C\left( \frac1R+\frac1{\sqrt{T}}+\sqrt{k}\right)
\left(1+\ln\frac{M}{u(x,t)}\right),\end{equation*}
for a suitable positive dimensional constant~$C$.
\end{thm}

This result has been extended
by Li Ma, Lin Zhao and Xianfa Song~\cite{MR2392508}
to the case of nonlinear equations, obtaining
the following structural result:

\begin{thm}[Theorem 7 in \cite{MR2392508}]\label{HEAT2}
Let~$\mathscr{M}$
be a complete Riemannian manifold of
dimension~$n$
with
$\mathrm{Ric}(\mathscr{M})\geq-k$, for some~$k\ge0$.
Let~$ u=u(x,t)$ be a positive solution of~$u_t=\Delta (F(u))$ in~$Q_{R,T}$, with~$F\in C^2(0,+\infty)$.
Assume that~$u\le M$ for some~$M>0$
and that
\begin{equation}\label{CONDMA0}
{\mbox{$F'(s)\in(0,K]$ for every~$s\in(0,M]$.}}\end{equation}
Let also~$G:(0,+\infty)\to\R$ be such that~$G'(s)=F'(s)/s$
for all~$s\in(0,M]$ and suppose that
\begin{equation}\label{CONDMA}
\begin{split}&
1-\frac{\sqrt{n}\,|F''(s)|\,s}{F'(s)}\geq\kappa>0,\\&
\xi-G(s)\geq\eta>0\\{\mbox{and }}\qquad&
2F'(s)-\frac{\sqrt{n}|F''(s)|s}{F'(s)}\big(\xi-G(s)\big)>0
\end{split}
\end{equation}
for every~$s\in(0,M]$, for suitable constants~$\kappa$, $\eta$, $\xi$.

Then,
there exists~$C>0$, depending only on~$n$,
$K$, $\kappa$ and $\eta$ such that,
for each~$(x,t)\in Q_{R/2,T/2}$,
\begin{equation}\label{Liyh2f9maliummorjfy}
\frac{|\nabla G(u(x,t))|}{\xi-G(u(x,t))}
\le C\left(\frac1R+\frac1{\sqrt{T}}+\sqrt{k}\right).
\end{equation}
\end{thm}

As detailed in Remark~8 in~\cite{MR2392508} (see also Appendix~\ref{APP-90}
here),
Theorem~\ref{HEAT2}
includes Theorem~\ref{HEAT1} as a special case, when~$F(s)=s$. Moreover, while conditions~\eqref{CONDMA} may look rather technical at a first glance, they are in fact
sufficiently general to treat several important nonlinear models such as the porous
medium equation~$u_t=\Delta u^p$ (see e.g.~\cite{MR2286292},
and also~\cite{MR2383484, MR3427985, MR3658720, MR3735744}
for the case of Riemannian manifolds)
with
\begin{equation}\label{RANGEP}
p\in\left(1-\frac1{\sqrt{n}},1\right].
\end{equation}
In this case, Theorem~\ref{HEAT2}
entails the following statement:

\begin{cor}[Corollary 9 in \cite{MR2392508}]\label{99}
Let~$\mathscr{M}$
be a complete Riemannian manifold of dimension~$2$
and~$3$,
with~$\mathrm{Ric}(\mathscr{M})\geq-k$, for some~$k\ge0$.
Let~$ u=u(x,t)$ be a positive solution of~$u_t=\Delta u^p$ in~$Q_{R,T}$,
with~$p$ as in~\eqref{RANGEP}.
Assume that~$u\le M$ for some~$M>0$.

Then,
there exists~$C>0$, depending only on~$n$ and~$p$ such that,
for each~$(x,t)\in Q_{R/2,T/2}$,
\begin{equation}\label{TIm94}
\frac{|\nabla u(x,t)|}{ u(x,t) }
\le C\left(\frac1R+\frac{M^{\frac{1-p}2}}{\sqrt{T}}+\sqrt{k}\right).
\end{equation}
\end{cor}

For the sake of precision, we observe that, strictly speaking,
in the original formulation given in~\cite{MR2392508},
Corollary~9 in~\cite{MR2392508} is not a direct consequence
of Theorem~7 in~\cite{MR2392508}, since the proof of the corollary presented
there does not rely merely on the statement of the theorem but rather
on a skillful modification of its proof: nevertheless,
it is possible to deduce the corollary directly from the results
that we will present here, as we point out in Appendix~\ref{APP-90}.

We recall that existence and uniqueness results
for the porous medium equation with~$p<1$ have been established in~\cite{MR797051}.
We also mention that universal pointwise estimates for porous medium equations
have been obtained in~\cite{MR524760}. In~\cite{MR2487898}
the authors prove, together with other gradient estimates on manifolds,
that one can derive from these universal pointwise estimates
also gradient estimates in the case of fast diffusion.

See also~\cite{MR1219716, MR1431005, MR1618694, MR1664888, MR1720778, MR1742035, MR2264255, MR3621772, 12345, 2019arXiv190304569C, ZU2, MR3633806,
MR2853544, MR2425752, MR2763753, MR3023250}
for parabolic estimates related to the results presented so far.

\medskip

The goal of this article is to enhance
Theorem~\ref{HEAT2} (and consequently Theorem~\ref{HEAT1}) in several directions:
\begin{itemize}
\item First of all, we replace the nonlinear operator~$\Delta(F(u))$ with the more general
nonlinear diffusive term $$a(x,t,u)\Delta(F(u)).$$ Even when~$F(s)=s$,
this improvement is interesting since it corresponds to allowing a heat equation in which the diffusion coefficient of the
substratum depends on space, time, and possibly also the temperature;
\item Moreover, we allow a source term depending on space, time, on the solution itself, and possibly
also on the gradient
and the Hessian of the solution;
\item In addition, instead of a local estimate, we obtain a global
estimate in~$Q_{R,T}$, depending on the parabolic data of the equation,
which recovers the universal estimate in~$Q_{R/2,T/2}$ as a byproduct.
\end{itemize}
To obtain our result, we will perform a number of rather involved and ad-hoc computations and exploit also the cut-off function method that was introduced
in~\cite{ZU2} to address global estimates. We also remark that, as far as we know, our results are new also in the case of nonlinear parabolic equations in the Euclidean space when~$k=0$.\medskip

Our result relies on suitable structural assumptions, that can be seen as natural counterparts of those in~\eqref{CONDMA}, and, to state clearly the estimates obtained, we now introduce precisely the mathematical framework in which we work.
\medskip

We consider the evolution equation
\begin{equation}\label{EQ}
u_{t}=a(x,t,u)\,\Delta(F(u))+H(x,t,u,\nabla u, D^2u)
\end{equation}
on a complete Riemannian manifold~$\mathscr{M}$
of dimension~$n$ and such that
\begin{equation}\label{RIC}
\mathrm{Ric}(\mathscr{M})\geq-k\end{equation}
for some $k\in\R$.
In this notation, $u=u(x,t)$, where~$x\in\mathscr{M}$
is the space variable and $t$ is the time variable.
As customary, the notation ``$\nabla$'' and~``$\Delta$''
is reserved, respectively, for the gradient and the Laplacian
in the space variable.
We suppose that equation~\eqref{EQ} is satisfied for
every~$(x,t)\in
Q_{R,T}\subset\mathscr{M}\times(-\infty,\infty)$,
where~$Q_{R,T}$ was introduced in~\eqref{QRT}.

We take~$a\in C^1(Q_{R,T}\times\R)$ with
\begin{equation}\label{a0bound}
a(x,t,s)\in [a_0,a_0^{-1}]\end{equation}
for all~$(x,t,s)\in Q_{R,T}\times\R$,
for some~$a_0\in(0,1)$.

We suppose that the solution~$u$ is smooth, positive and bounded, namely,
that for every~$(x,t)\in
Q_{R,T}$ we have~$u(x,t)\in(0,M]$, for some~$M>0$.

We suppose that~$F\in C^{2}(0,+\infty)$ with
\begin{equation}\label{FPRIM}
F'(s)>0
\end{equation}
and
\begin{equation}\label{kappapi}
1-\frac{\sqrt{n}\,|F''(s)|\,s}{F'(s)}\geq\kappa>0,\end{equation}
for all~$s\in(0,M]$,
for a positive constant~$\kappa$,
and that $H\in C^1
\big( Q_{R,T}\times\R\times\R^n\times \R^{n^2}\big)$.
With respect to the variables of~$H$, the ``gradient-Hessian''
coordinates in~$\R^n\times \R^{n^2}$ will be denoted by
\begin{equation}\label{COO}
(\omega,\Omega)=\big(
(\omega_i)_{i\in\{1,\dots,n\}},(\Omega_{ij})_{i,j\in\{1,\dots,n\}}\big).
\end{equation}
We also take~$s_0\in(0,+\infty)$ and define, for all~$s\in(0,M]$,
\begin{equation}\label{DEFG} G(s):=\int_{s_0}^s \frac{F'(h)}{h}\,dh,\end{equation}
and we assume that
\begin{equation}\label{XIPOS}
\xi-G(s)\geq\eta>0,
\end{equation}
that
\begin{equation}\label{SUPER}
\frac{F'(s)}{\xi-G(s)}\le\Gamma
\end{equation}
and that
\begin{equation}\label{CONDFOD}
2F'(s)-\frac{\sqrt{n}|F''(s)|s}{F'(s)}\big(\xi-G(s)\big)\geq0
\end{equation}
for all~$s\in(0,M]$, for a
suitable\footnote{It is interesting to remark that
conditions~\eqref{CONDMA0} and~\eqref{CONDMA}
are stronger than
conditions~\eqref{XIPOS}, \eqref{SUPER}
and~\eqref{CONDFOD}. In particular,
if~\eqref{CONDMA0} and~\eqref{CONDMA}
are satisfied, one can take~$\Gamma:=K/\eta$ in~\eqref{SUPER}.
On the other hand, as it will be apparent in Appendix~\ref{APP-90},
it is technically convenient to avoid
requesting assumption~\eqref{CONDMA0}
in order not to limit the potential of the general approach that we present.}
constant~$\xi$ and positive constants~$\eta$ and~$\Gamma$.

We introduce the structural constants
\begin{equation}\label{DEFmu}\begin{split}
\mu_1:=\sup_{(x,t)\in Q_{R,T}}&\Bigg(ka(x,t,u(x,t))F'(u(x,t))
+\frac{H(x,t,u(x,t),\nabla u(x,t), D^2u(x,t))\,F''(u(x,t))}{F'(u(x,t)
)}\\&\qquad\qquad+\partial_{u}H(x,t,u(x,t),\nabla u(x,t), D^2u(x,t))\\&\qquad\qquad
-\frac{H(x,t,u(x,t),\nabla u(x,t), D^2u(x,t))}{u(x,t)
}\\&\qquad\qquad+\frac{H(x,t,u(x,t),\nabla u(x,t), D^2u(x,t))F'(u(x,t))}{\big(\xi-G(u(x,t))\big)\;u(x,t)}\Bigg)_+\end{split}
\end{equation}
and
\begin{equation}\label{DEFgamma}
\gamma_1:=\sup_{(x,t)\in Q_{R,T}}\frac{F'(u(x,t))|\nabla H(x,t,u(x,t),\nabla u(x,t), D^2u
(x,t))|}{u(x,t)}.
\end{equation}
Let also
\begin{equation}\label{MUDUE}\begin{split}&
\mu_2:=\sup_{(x,t)\in Q_{R,T}}\big|\partial_{u}a(x,t,u(x,t))\big|\;
\Big|{\rm div}\big(F'(u(x,t))\nabla u(x,t)\big)\Big|\end{split}
\end{equation}
and
\begin{equation}\label{DEFMU22}
\mu:=\mu_1+\mu_2.\end{equation} We stress that~$\mu_2=0$,
and thus~$\mu=\mu_1$,
when~$a$ depends only on~$x$ and~$t$ (but is
independent of~$u$).

We also consider the quantities
\begin{equation}\label{GAMMA2}\begin{split}&
\gamma_2:=\sup_{(x,t)\in Q_{R,T}}
\frac{ F'(u(x,t))}{ u(x,t)}\big|\nabla a(x,t,u(x,t))\big|\;
\Big|{\rm div}\big(F'(u(x,t))\nabla u(x,t)\big)\Big|\end{split}
\end{equation}
and
\begin{equation}\label{GAMMA32}\begin{split}&
\gamma_3:=\sup_{(x,t)\in Q_{R,T}}
\frac{ F'(u(x,t))}{ u(x,t)}\Bigg(
|\nabla_{\omega} H(x,t,u(x,t),\nabla u(x,t),D^2u(x,t))|\,|D^2u(x,t)|\\&\qquad\qquad+
|D_{\Omega} H(x,t,u(x,t),\nabla u(x,t),D^2u(x,t))|\,|D^3 u(x,t)|\Bigg)\end{split}
\end{equation}
and we set
\begin{equation}\label{DEFMU33}
\gamma:=\gamma_1+\gamma_2+\gamma_3.\end{equation}
We remark that~$\gamma_2=\gamma_3=0$,
and thus~$\gamma=\gamma_1$,
when~$a$ depends only on~$t$ and~$u$ (but is
independent of the space variable), and~$H$ depends only on~$x$, $t$ and~$u$ (but is
independent of the gradient and of the Hessian of the solution).

Given~$\delta\in(0,T)$ and $\rho\in(0,R)$, we define
\begin{equation}\label{TUTTITE-0}
\begin{split}
&{\mathscr{C}}:=\sqrt{\mu}+\sqrt[3]{\gamma},\\
&{\mathscr{T}}:=\frac{1}{\sqrt{\delta}}\\{\mbox{and }}\qquad
&{\mathscr{S}}:=\frac{1}{\rho}+\frac{1}{\sqrt{\rho(R-\rho)}}+\frac{\sqrt[4]{k_+}}{\sqrt{\rho}}.
\end{split}
\end{equation}
We notice that~${\mathscr{C}}$, ${\mathscr{T}}$
and~${\mathscr{S}}$ are functions of~$(x,t)$.
Moreover, we set
\begin{equation}\label{sigtau}\begin{split}&
\tau_u:=\sup_{x\in B(x_0,R)}\frac{F'(u)|\nabla u|}{u(\xi-G(u))}(x,t_0-T),
\\{\mbox{and }}\qquad&
\sigma_u:=\sup_{{x\in\partial B(x_0,R)}\atop{t\in[t_0-T,t_0]}}\frac{F'(u)|\nabla u|}{u(\xi-G(u))}(x,t).
\end{split}\end{equation}

We also consider the functions
\begin{equation}\label{161}\begin{split}&
{\mathscr{B}}_1(x,t):=\chi_{B(x_0,R-\rho)}(x)\chi_{[t_0-T, t_0-T+\delta)}(t),\\&
{\mathscr{B}}_2(x,t):=\chi_{B(x_0,R)\backslash B(x_0,R-\rho)}(x)\chi_{[t_0-T+\delta,t_0]}(t),\\&
{\mathscr{B}}_3(x,t):=\chi_{B(x_0,R)\backslash B(x_0,R-\rho)}(x)\chi_{[t_0-T,t_0-T+\delta)}(t)\\{\mbox{and }}\qquad&
{\mathscr{I}}(x,t):=\chi_{B(x_0,R-\rho)}(x)\chi_{[t_0-T+\delta,t_0]}(t).
\end{split}\end{equation}
As customary, we used here the standard notation
for the characteristic function of a set~$S$, that is
$$ \chi_S(z):=\begin{cases}
1 & {\mbox{ if }}z\in S,\\0&{\mbox{ otherwise.}}
\end{cases}
$$
Also, given a constant~$C>0$ (to be appropriately chosen conveniently
large in the following)
we define
\begin{equation}\label{iotabeta}\begin{split}
\beta_1\,&:=
\tau_u +\min\left\{\sigma_u,\,C{{\mathscr{S}}}\right\},\\
\beta_2\,&:=
\sigma_u +\min\left\{\tau_u,\,C{{\mathscr{T}}}\right\},\\
\beta_3\,&:=\sigma_u +\tau_u\\ {\mbox{and }}\qquad
\iota\,&:=
\min\left\{\sigma_u+\tau_u,\,C({{\mathscr{T}}}+{{\mathscr{S}}})\right\}
.\end{split}\end{equation}
Let also
\begin{equation}\label{ie75v8b76vb502428v0}
\mathscr{Z}:=
\beta_1\,{\mathscr{B}}_1
+\beta_2\,{\mathscr{B}}_2
+\beta_3\,{\mathscr{B}}_3
+\iota\,{\mathscr{I}}.
\end{equation}

With this notation, the main result of this
paper is the following global gradient estimate,
valid in all the domain where a parabolic nonlinear equation
holds true:

\begin{thm}\label{D2}
Let~$\mathscr{M}$ be a
complete Riemannian manifold
of dimension~$n$  satisfying~\eqref{RIC}.
Let~$a\in C^1(Q_{R,T}\times\R)$ satisfying~\eqref{a0bound}
and~$F\in C^{2}(0,+\infty)$ satisfying~\eqref{FPRIM},
\eqref{kappapi}, \eqref{XIPOS}, \eqref{SUPER}
and~\eqref{CONDFOD}.

Let~$u$ be a positive, bounded and smooth solution of the evolution equation~\eqref{EQ}
in~$Q_{R,T}$.

Then,  there exists a constant~$C>0$ depending only on~$n$, $\Gamma$, $\kappa$, $a_{0}$ and~$\eta$ such that the following estimate holds true:
\begin{equation*}
G'(u(x,t))\,|\nabla u(x,t)|\leq \Big(C
\mathscr{C}+
\mathscr{Z}(x,t)\Big)\,
\Big(\xi-G(u(x,t))\Big)\quad{\mbox{ for all }}
(x,t)\in Q_{R,T}.
\end{equation*}
Here we used the notation in~\eqref{DEFG}, \eqref{TUTTITE-0}
and~\eqref{ie75v8b76vb502428v0}.
\end{thm}

Interestingly, Theorem~\ref{D2} includes several recent results
in the literature as a special case. For instance, the particular choice
\begin{equation}\label{RCNm:0orlfkk33}
{\mbox{$F(s):=s$, $a:=1$ and~$H:=H(x,t,u)$,}}
\end{equation}
corresponding to the equation~$\partial_t u=\Delta u+H(x,t,u)$,
taking~$\xi:=1$, $\eta:=1$, $s_0:=M$,
produces in~\eqref{DEFmu}
the quantity
\begin{equation*}
\mu_1=\sup_{(x,t)\in Q_{R,T}}\Bigg(k
+\partial_{u}H(x,t,u(x,t))
-\frac{H(x,t,u(x,t))}{u(x,t)
}+\frac{H(x,t,u(x,t))}{\left(1+\log\frac{M}{u(x,t)}\right)\,u(x,t)}\Bigg)_+,
\end{equation*}
which coincides with the quantity in~(1.6)
of~\cite{ZU2}; similarly, in such a case, in~\eqref{DEFgamma}
we find
\begin{equation*}
\gamma_1=\sup_{(x,t)\in Q_{R,T}}\frac{|\nabla H(x,t,u(x,t))|}{u(x,t)},
\end{equation*}
which coincides with the quantity in~(1.4)
of~\cite{ZU2};
also, in~\eqref{MUDUE}, \eqref{GAMMA2} and~\eqref{GAMMA32}
one finds~$\mu_2=
\gamma_2=\gamma_3=0$, therefore Theorem~\ref{D2}
here recovers, in the special setting of~\eqref{RCNm:0orlfkk33},
the result given in Theorem~1.1 of~\cite{ZU2}.
\medskip

Differently from the previous literature,
our general framework comprises, as a particular
case, the equation~$u_t=u\Delta u+g(u)$ which
models the spread of an epidemic in a closed population
without remotion
and is often
used as a prototype for complicated
and sometimes pathological behavior of the solutions,
see~\cite{MR859613, MR897437}.
\medskip

It is also interesting to comment
on the structure of the estimate obtained in Theorem~\ref{D2},
and especially on the dependence of the bound obtained
by the quantities~$\mu$ and~$\gamma$.
Specifically, being a gradient estimate,
one would like the terms on the right hand side of the estimate
to be independent of the derivatives of the solution,
while, at a first glance, these quantities may depend
on the derivatives up to order three. Nevertheless:
\begin{itemize}
\item The dependence of~$\mu_1$ and~$\gamma_1$
in~\eqref{DEFmu} and~\eqref{DEFgamma}
on the derivatives of~$u$ only occurs via
the source term~$H$: in particular,
if~$H$ and its derivatives are uniformly bounded, then~$\mu_1$ and~$\gamma_1$
can be bounded independently on the derivatives of~$u$;
\item The quantities~$\mu_2$ and~$\gamma_2$
in~\eqref{MUDUE} and~\eqref{GAMMA2}
depend on the derivatives of the solution up to the second order,
but they vanish if the diffusion coefficient~$a$ is either constant
or depends only on time;
\item The quantity~$\gamma_3$ in~\eqref{GAMMA32}
depends on the derivatives up to the third order of the solution,
but it vanishes if the source term~$H$ does not depend on the
derivatives of the solution.
\end{itemize}
That is: on the one hand, in its general form,
under additional bounds on the derivatives of the solution,
the estimate in Theorem~\ref{D2} can be considered as
a pointwise estimate at any~$(x,t)\in Q_{R,T}$;
on the other hand, for the special (but still extremely general)
case given by the equation
$$ \partial_t u=a(t)\,\Delta(F(u))+H(x,t,u),$$
then the structural quantities~$\mu$ and~$\gamma$ can
be bounded independently from the derivatives of the solution,
reducing to
\begin{eqnarray*}
&&\mu=
\sup_{(x,t)\in Q_{R,T}}\Bigg(ka(t)F'(u(x,t))
+\frac{H(x,t,u(x,t))\,F''(u(x,t))}{F'(u(x,t))}\\&&\qquad\qquad+\partial_{u}H(x,t,u(x,t))
-\frac{H(x,t,u(x,t))}{u(x,t)
}+\frac{H(x,t,u(x,t))F'(u(x,t))}{\big(\xi-G(u(x,t))\big)\;u(x,t)}\Bigg)_+\\
{\mbox{and }}&&\gamma=
\sup_{(x,t)\in Q_{R,T}}\frac{F'(u(x,t))|\nabla H(x,t,u(x,t))|}{u(x,t)}.
\end{eqnarray*}
\medskip

The relevant structural quantities
in case of an equation of the type
\begin{equation}\label{APPEPSE}
\partial_t u=a(t)\Delta u^p+\varepsilon\,|\nabla u|^{q}
\end{equation}
with~$\varepsilon>0$
will be discussed, as an exemplifying situation,
in Appendix~\ref{APPEPS}.
\medskip

Furthermore,
one deduces from the global estimate of
Theorem~\ref{D2} a local estimate in~$Q_{R/2,T/2}$, according to the following result.

\begin{cor}\label{noqwmefef004}
Let~$\mathscr{M}$ be a
complete Riemannian manifold
of dimension~$n$  satisfying~\eqref{RIC}.
Let~$a\in C^1(Q_{R,T}\times\R)$ satisfying~\eqref{a0bound}
and~$F\in C^{2}(0,+\infty)$ satisfying~\eqref{FPRIM},
\eqref{kappapi}, \eqref{XIPOS}, \eqref{SUPER}
and~\eqref{CONDFOD}.

Let~$u$ be a positive, bounded and smooth solution of the evolution equation~\eqref{EQ}
in~$Q_{R,T}$.

Then,  there exists a constant~$C>0$
depending only on~$n$, $\Gamma$, $\kappa$, $a_{0}$
and~$\eta$ such that the following estimate holds true: for every~$(x,t)\in Q_{R/2,T/2}$,
\begin{equation}\label{7uNS:034r}
G'(u(x,t))\,|\nabla u(x,t)|\leq C
\left( \sqrt{\mu}+\sqrt[3]{\gamma}+
\frac{1}{R}+\frac1{\sqrt{T}}+\frac{\sqrt[4]{k_+}}{\sqrt{R}}
\right)\Big(\xi-G(u(x,t))\Big).\end{equation}
Here we used the notation in~\eqref{DEFG}, \eqref{DEFMU22}
and~\eqref{DEFMU33}.
\end{cor}

We also stress that when~$H:=0$ and~$a:=1$, then~$\mu=k_+$
and~$\gamma=0$, therefore
Corollary~\ref{noqwmefef004} contains
Theorem~\ref{HEAT2} (that is,
Theorem 7 in \cite{MR2392508})
as a special case.
In addition,
it also contains Corollary~\ref{99}
(that is,
Corollary 9 in \cite{MR2392508}) as a particular subcase,
as observed in Appendix~\ref{APP-90}.
\medskip

It is also interesting to compare
the statements of
Theorem~\ref{D2} and Corollary~\ref{noqwmefef004}. Evidently,
the estimate obtained in Theorem~\ref{D2}
is global, since it is valid in the whole of the domain
where the equation is satisfied. For this,
the estimate obtained in Theorem~\ref{D2} necessarily
must take into account the ``parabolic data'' of the equation,
which are encoded in the quantities~$\tau_u$ and~$\sigma_u$
defined in~\eqref{sigtau}. On the other hand,
the estimate obtained in Corollary~\ref{noqwmefef004}
holds true only in a subdomain, but then it becomes independent
of the ``parabolic data'' of the equation and relies only on the structural
functions of the equation and on the geometry of the domain.\medskip

We emphasize that the general estimate in
Theorem~\ref{D2} is stronger than the one in Corollary~\ref{noqwmefef004}
even if one reduces to~$Q_{R/2,T/2}$, since
one can also deduce from it that, in~$Q_{R/2,T/2}$,
\begin{equation}\label{ojscdnc98475ty}
G'(u(x,t))\,|\nabla u(x,t)|\leq C
\left( \sqrt{\mu}+\sqrt[3]{\gamma}+\sigma_u+\tau_u
\right)\Big(\xi-G(u(x,t))\Big),\end{equation}
which is a sharper estimate than the one in~\eqref{7uNS:034r}
when the data of the equation are particularly convenient to make~$\tau_u$ and~$\sigma_u$
sufficiently small. That is,
while the estimate in Corollary~\ref{noqwmefef004} has
the advantages of being easier to read
and ``universal'' (i.e., not depending on the boundary data of the equation),
the estimate in Theorem~\ref{D2} is more precise, since
it allows one to possibly recall the boundary data
in order to achieve a sharper result.\medskip

In any case,
to the best of our knowledge, Theorem~\ref{D2}
is the first global estimate for nonlinear parabolic operators,
even in the case of porous medium equation with no source terms,
and also the local version in Corollary~\ref{noqwmefef004}
is the first local estimate to take into account
general porous medium equations with source terms;
besides, the alternative estimate in~\eqref{ojscdnc98475ty}
is the first occurrence in which an improved estimate for these parabolic
equations driven by nonlinear operators
is obtained
thanks to the boundary data.
Moreover, the results obtained are new even in the Euclidean setting.\medskip

We also remark that suitable Liouville-type
results can be easily deduced from our main estimates:
as an example, we provide a rigidity result in Appendix~\ref{DFDF}
that relies on Corollary~\ref{noqwmefef004}.
See also Theorems~1.3 and~1.5 for related
Liouville-type results for porous media equations.\medskip

The rest of this paper is organized as follows.
Section~\ref{MSF:DOFOF} presents the computations
related to a suitable auxiliary function that
will be used to deduce the main results from the maximum principle.
In Section~\ref{KJSMD93sad45}
we develop the necessary calculations
to localize the problem by using suitable cut-off functions
in space and time.
Section~\ref{KJS-M-D93sad45}
contains the proof of
Theorem~\ref{D2} and
Section~\ref{8uCODroldd} provides
the one of Corollary~\ref{noqwmefef004}.

\section{An auxiliary barrier}\label{MSF:DOFOF}

A common procedure in the theory of elliptic and parabolic equations
is to introduce a suitable auxiliary function (that will be denoted by~$w$
in our context) which satisfies a convenient equation; with this,
an appropriate use of the maximum principle provides estimates
on the auxiliary function, which can
be traced back to the original solution.
To implement this technique in our framework, we argue as follows.

Given~$G$ as in~\eqref{DEFG}, for all~$r\in\R$ we define
\begin{equation}\label{DEFg}
g(r):=G(e^r)\end{equation}
and
\begin{equation}\label{DEFlambda}
\lambda(r):=\frac{g'(r)}{\xi-g(r)}-1+\frac{\sqrt{n}|g''(r)|}{2g'(r)}.\end{equation}
It is interesting to observe that, by~\eqref{DEFG},
\begin{equation}\label{1.8bis}
G'(r)=\frac{F'(r)}{r}\qquad{\mbox{and}}\qquad
G''(r)=\frac{F''(r)}{r}-\frac{F'(r)}{r^{2}}.\end{equation}
In addition, since, by~\eqref{DEFg}, we know that~$
g'(r)=e^r\,G'(e^r)$, we deduce from
the assumption~\eqref{FPRIM} on~$F'$ and~\eqref{1.8bis} that
\begin{equation}\label{101}
g'(r)=F'(e^r)>0 \qquad{\mbox{and}}\qquad
g''(r)=e^r F''(e^r).
\end{equation}
Also,
given $u$ as in the statement of Theorem~\ref{D2}, we set
\begin{equation}\label{DEFv}
v(x,t):=\ln u(x,t)\end{equation}
and
\begin{equation}\label{DEFw}
w(x,t):=|\nabla \ln (\xi-g(v(x,t))|^{2}=\frac{|\nabla (g(v(x,t))|^{2}}{(\xi-g(v(x,t)))^{2}}.\end{equation}
We stress that~$w$ is well defined, since
\begin{equation}\label{xig}
\xi-g(v(x,t))=
\xi-G(u(x,t))\ge\eta>0,\end{equation} thanks to~\eqref{XIPOS}.

Also, as usual, the notation~$(
u,\nabla u, D^2u)$ will be used as short for~$(u(x,t),\nabla u(x,t), D^2u(x,t))$.
Furthermore, the notation~$g'$ stands for the derivative of~$g(r)$ with
respect to~$r$, hence~$g'(v)$ is a short notation for~$g'(v(x,t))$.
To clarify this framework, let us point out that
\begin{equation}\label{100}\begin{split}&
\nabla(g(v(x,t))=g'(v(x,t))\nabla v(x,t)\\{\mbox{and }}\qquad&
\nabla(g'(v(x,t))=g''(v(x,t))\nabla v(x,t)=\frac{g''(v(x,t))\;\nabla(g(v(x,t))}{g'(v(x,t))}.
\end{split}
\end{equation}
We recall that the latter denominator is nonzero, thanks to~\eqref{101}.
With this setting, we can state the main result of this section as follows:

\begin{lem}\label{Dl1} Let~$u$ be as in Theorem~\ref{D2}. Then, in~$Q_{R,T}$,
\begin{equation}
\label{LU190h010}
\frac{ag'\Delta w-w_t}{2}\ge
a \kappa (\xi-g) w^2
+a\lambda
\left\langle\nabla w,\nabla g\right\rangle-\mu w
-\frac{\gamma\,|\nabla g|}{(\xi-g)^2},
\end{equation}
where~$g$ is a short notation for~$g(v(x,t))$.
Here, $\kappa$ and~$\xi$ are given in assumptions~\eqref{kappapi}
and~\eqref{XIPOS}, $\lambda$ in~\eqref{DEFlambda}, $\mu$
in~\eqref{DEFMU22} and~$\gamma$ in~\eqref{DEFMU33}.
\end{lem}

\begin{proof}
We note that
\begin{equation}\label{1.8pre}
\Delta(F(u))=F'(u)\Delta u+F''(u)|\nabla u|^2.
\end{equation}
Hence, by~\eqref{EQ},
\begin{equation}\label{1.8}
u_t=aF'(u)\Delta u+aF''(u)|\nabla u|^{2}+H.
\end{equation}
Also, by~\eqref{DEFv},
\begin{equation}\label{BB}
v_t=\frac{ u_t}u\qquad{\mbox{and}}\qquad
\nabla v=\frac{\nabla u}u,
\end{equation}
whence
\begin{equation}\label{AQU} \frac{G'(u)}{u}\,|\nabla u|^2=\frac{\nabla (G(u))\cdot\nabla u}{u}
=\nabla (G(u))\cdot\nabla v.\end{equation}
To ease the notation, we write~$\nabla G$ as a short notation for~$\nabla (G(u))=\nabla (G(u(x,t)))$ (of course, no confusion
should arise with~$\nabla G(u(x,t))$).
Accordingly, exploiting~\eqref{1.8bis},
\eqref{1.8pre}, \eqref{1.8}, \eqref{BB}
and~\eqref{AQU},
\begin{equation*}\begin{split}
v_t\,&=\frac{aF'(u)\Delta u+aF''(u)|\nabla u|^{2}+H}{u}
\\&=aG'(u)\Delta u+a\left(G''(u)+\frac{F'(u)}{u^{2}}\right)|\nabla u|^{2}+\frac{H}{u}\\
&=aG'(u)\Delta u+a\left(G''(u)+\frac{G'(u)}{u}\right)|\nabla u|^{2}+\frac{H}{u}\\
&=a\Delta(G(u))+a\left\langle\nabla G,\nabla v\right\rangle+\frac{H}{u}.
\end{split}\end{equation*}
Thus, recalling~\eqref{DEFg} and~\eqref{DEFv}, we can write~$G(u)=g(v)$ and thereby obtain that
\begin{equation}\label{89:9}
v_t=a\Delta(g(v))+a\left\langle\nabla g ,\nabla v\right\rangle+\frac{H}{u}.
\end{equation}
where~$\nabla g$ is short for~$\nabla (g( v ))=\nabla( g(v(x,t)))$.

To ease the notation, we also write~$g_t$
to mean~$\partial_{t}(g(v))$.
As a consequence, by~\eqref{89:9} we have that
\begin{equation}\label{DD4}\begin{split}
g_t=g'(v)v_t\,&=ag'(v)\Delta(g(v))+
a\left\langle\nabla g ,g'(v)\nabla v\right\rangle+\frac{g'(v)H}{u}\\
&=ag'(v)\Delta(g(v))+a|\nabla g |^{2}+\frac{g'(v)H}{u}.
\end{split}\end{equation}

Now we observe that, by~\eqref{DEFw},
\begin{equation}\label{DD5}
\nabla w
=\frac{\nabla|\nabla g|^2}{(\xi-g)^{2}}+2\frac{|\nabla g|^2\nabla g}{(\xi-g)^{3}}.
\end{equation}
Moreover, we have that
\begin{equation}\label{DD6}
{\rm div}\,\left(\frac{\nabla|\nabla g|^2}{(\xi-g)^{2}}\right)=
\frac{\Delta|\nabla g|^2}{(\xi-g)^{2}}+
\frac{2\left\langle\nabla|\nabla g|^2,\nabla g\right\rangle}{(\xi-g)^{3}}.
\end{equation}
In addition,
\begin{equation*}
{\rm div}\,\left(\frac{|\nabla g|^2\nabla g}{(\xi-g)^{3}}\right)=
\frac{\left\langle\nabla |\nabla g|^2,\nabla g\right\rangle}{(\xi-g)^{3}}
+
\frac{|\nabla g|^2\Delta g}{(\xi-g)^{3}}+
\frac{3|\nabla g|^4}{(\xi-g)^{4}}.
\end{equation*}
{F}rom this, \eqref{DD5} and~\eqref{DD6}, we deduce that
\begin{equation}\label{DD7}
\begin{split}&
\Delta w\\=\;&
\frac{\Delta|\nabla g|^2}{(\xi-g)^{2}}+
\frac{2\left\langle\nabla|\nabla g|^2,\nabla g\right\rangle}{(\xi-g)^{3}}
+
\frac{2\left\langle\nabla |\nabla g|^2,\nabla g\right\rangle}{(\xi-g)^{3}}
+
\frac{2|\nabla g|^2\Delta g}{(\xi-g)^{3}}+
\frac{6|\nabla g|^4}{(\xi-g)^{4}}
\\=\;&\frac{\Delta|\nabla g|^2}{(\xi-g)^2}+\frac{4\left\langle\nabla|\nabla g|^2,\nabla g\right\rangle}{(\xi-g)^3}
+\frac{2|\nabla g|^2\Delta g}{(\xi-g)^3}+\frac{6|\nabla g|^4}{(\xi-g)^4}.
\end{split}\end{equation}
Besides, using~\eqref{DEFw} and~\eqref{DD4}, we find that
\begin{equation}\label{DD8}
\aligned
w_t&=\frac{2\left\langle\nabla g,\nabla g_t\right\rangle}{(\xi-g)^2}+\frac{2|\nabla g|^2g_t}{(\xi-g)^3}
\\&=\frac{2\left\langle\nabla g,\nabla\big(ag'\Delta g\big)\right\rangle}{(\xi-g)^2}
+\frac{2\left\langle\nabla g,\nabla(a|\nabla g|^2)\right\rangle}{(\xi-g)^2}
+\frac{2\left\langle\nabla g,\nabla\left(\frac{g'(v)H}{u}\right)\right\rangle}{(\xi-g)^2}
+\frac{2ag'\Delta g|\nabla g|^2}{(\xi-g)^3}\\
&~~~~+\frac{2a|\nabla g|^4}{(\xi-g)^3}
+\frac{2|\nabla g|^2\left(\frac{g'(v)H}{u}\right)}{(\xi-g)^3}.
\endaligned
\end{equation}
In light of~\eqref{100}, we also have that
\begin{equation}\label{DD9}
\begin{split}&
\left\langle\nabla g,\nabla\big(ag'\Delta g\big)\right\rangle\\
=\,&ag'\left\langle\nabla g,\nabla\Delta g\right\rangle
+a\Delta g\left\langle\nabla g,\nabla g'\right\rangle
+ g'\Delta g\left\langle\nabla g,\nabla a\right\rangle
+g'\Delta g\left\langle\nabla g,\partial_u a\nabla u\right\rangle
\\
=\,&ag'\left\langle\nabla g,\nabla\Delta g\right\rangle
+\frac{ag''\Delta g}{g'}\,|\nabla g|^2
+g'\Delta g\left\langle\nabla g,\nabla a\right\rangle
+\partial_u a\, g'\Delta g\left\langle\nabla g,\nabla u\right\rangle.\end{split}
\end{equation}

Now, using~\eqref{101} and~\eqref{DEFv},
we see that~$g''(v)=F''(u)u$.
Thus, recalling the coordinate notation in~\eqref{COO}, and making also use of~\eqref{100}
and~\eqref{BB},
we find that
\begin{equation}\label{TT}
\aligned
\nabla\left(\frac{g'(v)H}{u}\right)&=\frac{g''(v)H\nabla v}{u}+\frac{g'(v)\nabla H}{u}+\frac{g'(v)\partial_{u}H\nabla u}{u}-\frac{g'(v)H\nabla u}{u^{2}}+\Upsilon_1\\
&=\frac{F''(u)H\nabla g }{F'(u)}+\frac{F'(u)\nabla H}{u}+\partial_{u}H\nabla g -\frac{H\nabla g }{u}+\Upsilon_1,
\endaligned
\end{equation}
where
\begin{equation}\label{UPSI1} \Upsilon_1:=\frac{g'(v)}{u}\left(\sum_{i=1}^n\partial_{\omega_i} H\nabla\partial_{x_i} u+\sum_{i,j=1}^n\partial_{\Omega_{ij}} H
\nabla\partial_{x_i x_j}^2 u
\right).\end{equation}
This observation,
together with~\eqref{DD8},
\eqref{DD9} and~\eqref{TT}, yields that
\begin{equation}\nonumber
\aligned
w_t&=\frac{2ag'\left\langle\nabla g,\nabla\Delta g\right\rangle}{(\xi-g)^2}
+\frac{\frac{2ag''\Delta g}{g'}\,|\nabla g|^2}{(\xi-g)^2}
+\frac{2\left\langle\nabla g,\nabla(a|\nabla g|^2)\right\rangle}{(\xi-g)^2}
+\frac{2F''(u)H|\nabla g|^2}{F'(u)(\xi-g)^2}\\
&~~~~+\frac{2F'(u)\left\langle\nabla g,\nabla H\right\rangle}{u(\xi-g)^2}
+\frac{2\partial_{u}H|\nabla g|^2}{(\xi-g)^2}
-\frac{2H|\nabla g|^2}{u(\xi-g)^2}\\
&~~~~+\frac{2ag'\Delta g|\nabla g|^2}{(\xi-g)^3}
+\frac{2a|\nabla g|^4}{(\xi-g)^3}+\frac{2|\nabla g|^2\left(\frac{g'(v)H}{u}\right)}{(\xi-g)^3}+\Upsilon_2\\
&=\frac{2ag'\left\langle\nabla g,\nabla\Delta g\right\rangle}{(\xi-g)^2}
+\frac{2ag''\Delta g|\nabla g|^2}{g'(\xi-g)^2}
+\frac{2\left\langle\nabla g,\nabla(a|\nabla g|^2)\right\rangle}{(\xi-g)^2}
+\frac{2F''(u)H|\nabla g|^2}{F'(u)(\xi-g)^2}\\
&~~~~+\frac{2F'(u)\left\langle\nabla g,\nabla H\right\rangle}{u(\xi-g)^2}
+\frac{2\partial_{u}H|\nabla g|^2}{(\xi-g)^2}
-\frac{2H|\nabla g|^2}{u(\xi-g)^2}\\
&~~~~+\frac{2ag'\Delta g|\nabla g|^2}{(\xi-g)^3}
+\frac{2a|\nabla g|^4}{(\xi-g)^3}+\frac{2F'(u)H|\nabla g|^2}{u(\xi-g)^3}+\Upsilon_2
,\endaligned
\end{equation}
where
\begin{equation}\label{UPSIL2} \Upsilon_2:=\frac{2\left\langle \nabla g,\Upsilon_1\right\rangle}{(\xi-g)^2}
+\frac{2g'\Delta g\,\left\langle \nabla g,\nabla a+\partial_{u}a\nabla u\right\rangle}{(\xi-g)^2}.\end{equation}
This and~\eqref{DD7}, after the cancellation of the term~$\frac{2ag'\Delta g\,|\nabla g|^2}{(\xi-g)^3}$, give that
\begin{equation}\label{DD10}
\aligned
ag'\Delta w-w_t&=\frac{ag'\Delta|\nabla g|^2}{(\xi-g)^2}+\frac{4ag'\left\langle\nabla|\nabla g|^2,\nabla g\right\rangle}{(\xi-g)^3}
+\frac{6ag'|\nabla g|^4}{(\xi-g)^4}-\frac{2ag'\left\langle\nabla g,\nabla\Delta g\right\rangle}{(\xi-g)^2}\\
&~~~~-\frac{2F''(u)H|\nabla g|^2}{F'(u)(\xi-g)^2}
-\frac{2F'(u)\left\langle\nabla g,\nabla H\right\rangle}{u(\xi-g)^2}
-\frac{2\partial_{u}H|\nabla g|^2}{(\xi-g)^2}\\
&~~~~+\frac{2H|\nabla g|^2}{u(\xi-g)^2}
-\frac{2ag''\Delta g|\nabla g|^2}{g'(\xi-g)^2}
-\frac{2\left\langle\nabla g,\nabla(a|\nabla g|^2)\right\rangle}{(\xi-g)^2}
-\frac{2a|\nabla g|^4}{(\xi-g)^3}\\
&~~~~-\frac{2F'(u)H|\nabla g|^2}{u(\xi-g)^3}-\Upsilon_2.
\endaligned
\end{equation}
\par
Now we recall the Bochner's formula, according to which
$$ \Delta {\bigg (}{\frac {|\nabla g|^{2}}{2}}{\bigg )}=\left\langle \nabla \Delta g,\nabla g\right\rangle +|D^2 g|^{2}+{\mbox{Ric}}(\nabla g,\nabla g).$$
This and the Ricci curvature assumption in~\eqref{RIC} entail that
\begin{eqnarray*}
 \Delta|\nabla g|^{2}-2\left\langle \nabla \Delta g,\nabla g\right\rangle =
2|D^2 g|^{2}+2{\mbox{Ric}}(\nabla g,\nabla g)\ge
2|D^2 g|^{2}-2k\,|\nabla g|^2
,\end{eqnarray*}
where, as customary, the norm of a matrix is taken to be the square root
of the sum of the squares of its entry.

Plugging this information in~\eqref{DD10}, we conclude that
\begin{equation}\label{DD11}
\aligned
ag'\Delta w-w_t&\ge\frac{2ag'|D^2 g|^{2}}{(\xi-g)^2}
-\frac{2ag'k|\nabla g|^2}{(\xi-g)^2}
+\frac{4ag'\left\langle\nabla|\nabla g|^2,\nabla g\right\rangle}{(\xi-g)^3}
+\frac{6ag'|\nabla g|^4}{(\xi-g)^4}\\
&~~~~-\frac{2F''(u)H|\nabla g|^2}{F'(u)(\xi-g)^2}
-\frac{2F'(u)\left\langle\nabla g,\nabla H\right\rangle}{u(\xi-g)^2}
-\frac{2\partial_{u}H|\nabla g|^2}{(\xi-g)^2}\\
&~~~~+\frac{2H|\nabla g|^2}{u(\xi-g)^2}
-\frac{2ag''\Delta g|\nabla g|^2}{g'(\xi-g)^2}
-\frac{2\left\langle\nabla g,\nabla(a|\nabla g|^2)\right\rangle}{(\xi-g)^2}
-\frac{2a|\nabla g|^4}{(\xi-g)^3}\\
&~~~~-\frac{2F'(u)H|\nabla g|^2}{u(\xi-g)^3}-\Upsilon_2\\
&\ge \frac{2ag'|D^2 g|^{2}}{(\xi-g)^2}
+\frac{4ag'\left\langle\nabla|\nabla g|^2,\nabla g\right\rangle}{(\xi-g)^3}
+\frac{6ag'|\nabla g|^4}{(\xi-g)^4}
-\frac{2\mu_1|\nabla g|^2}{(\xi-g)^2}\\
&~~~~-\frac{2\gamma_1\,|\nabla g|}{(\xi-g)^2}
-\frac{2ag''\Delta g|\nabla g|^2}{g'(\xi-g)^2}-\frac{2\left\langle\nabla g,\nabla(a|\nabla g|^2)\right\rangle}{(\xi-g)^2}
-\frac{2a|\nabla g|^4}{(\xi-g)^3}-\Upsilon_2,
\endaligned
\end{equation}
where the definitions of~$\mu_1$ and~$\gamma_1$ in~\eqref{DEFmu}
and~\eqref{DEFgamma} have been exploited.

It is now convenient to define
\begin{equation}\label{ZETA:SDE}
\zeta:=\left(\sqrt{n}\frac{|g''|}{g'}(\xi-g)\right)^{-1}.
\end{equation}
We point out that~$\zeta\in(0,+\infty]$, due to~\eqref{101}
and~\eqref{xig}.
By the Cauchy-Schwarz inequality, if~$g''\ne0$ (hence~$\zeta\ne+\infty$),
one has that
\begin{equation}\label{SM:osd}
2\Bigg|\frac{g''\Delta g|\nabla g|^2}{g'(\xi-g)^2}\Bigg|\le\zeta\left(\frac{g''}{g'}\right)^{2}(\Delta g)^{2}+\frac{|\nabla g|^4}{\zeta(\xi-g)^4}
\le\zeta n\left(\frac{g''}{g'}\right)^{2}|D^{2}g|^{2}+\frac{|\nabla g|^4}{\zeta(\xi-g)^4}.\end{equation}
We will now make use of~\eqref{SM:osd}
at all points, with the convention that, since the left
hand side of~\eqref{SM:osd} vanishes when~$g''=0$,
the terms involving~$\zeta$ can simply be neglected
in the forthcoming computations. In this sense,
putting together~\eqref{DD11} and~\eqref{SM:osd} we see that
\begin{equation}\label{DD12}
\begin{split}
ag'\Delta w-w_t&\ge \Bigg[2g'-\zeta n\left(\frac{g''}{g'}\right)^{2}(\xi-g)^2\Bigg]\frac{a|D^2 g|^{2}}{(\xi-g)^2}
+\frac{4ag'\left\langle\nabla|\nabla g|^2,\nabla g\right\rangle}{(\xi-g)^3}
\\&\qquad+\Big(6g'-\frac{1}{\zeta} \Big)\frac{a|\nabla g|^4}{(\xi-g)^4}
-\frac{2\left\langle\nabla g,\nabla(a|\nabla g|^2)\right\rangle}{
(\xi-g)^2}\\&\qquad
-\frac{2a|\nabla g|^4}{(\xi-g)^3}-\frac{2\mu_1|\nabla g|^2}{(\xi-g)^2}
-\frac{2\gamma_1\,|\nabla g|}{(\xi-g)^2}-\Upsilon_2.
\end{split}
\end{equation}
Furthermore, in light of~\eqref{DD5},
\begin{equation}\nonumber
\left\langle\nabla w,\nabla g\right\rangle=
\frac{\left\langle\nabla |\nabla g|^2,\nabla g\right\rangle}{(\xi-g)^2}+\frac{2|\nabla g|^4}{(\xi-g)^3},
\end{equation}
and, as a result,
\begin{eqnarray*}&&
\frac{4ag'\left\langle\nabla |\nabla g|^2,\nabla g\right\rangle}{(\xi-g)^3}
+\frac{6ag'|\nabla g|^4}{(\xi-g)^4}\\&
=&
\left(
\frac{2ag'\left\langle\nabla |\nabla g|^2,\nabla g\right\rangle}{(\xi-g)^3}
+\frac{4ag'|\nabla g|^4}{(\xi-g)^4}\right)+\left(
\frac{2ag'\left\langle\nabla |\nabla g|^2,\nabla g\right\rangle}{(\xi-g)^3}
+\frac{2ag'|\nabla g|^4}{(\xi-g)^4}
\right)\\&
=&\frac{2ag'\left\langle\nabla w,\nabla g\right\rangle}{\xi-g}+
\frac{2ag'\left\langle\nabla |\nabla g|^2,\nabla g\right\rangle}{(\xi-g)^3}
+\frac{2ag'|\nabla g|^4}{(\xi-g)^4}
.\end{eqnarray*}
The previous
two identities, combined with~\eqref{DD12}, yield that
\begin{equation}\label{7tyfv6erdftwcgwuy2w}
\aligned
ag'\Delta w-w_t&\ge \Bigg[2g'-\zeta n\left(\frac{g''}{g'}\right)^{2}(\xi-g)^2
\Bigg]\frac{a|D^2 g|^{2}}{(\xi-g)^2}
+\frac{2ag'\left\langle\nabla|\nabla g|^2,\nabla g\right\rangle}{(\xi-g)^3}
\\&~~~~+\left(2g'-\frac{1}{\zeta} \right)\frac{a|\nabla g|^4}{(\xi-g)^4}
+\frac{2a|\nabla g|^4}{(\xi-g)^3}+2a\left(\frac{g'}{\xi-g}-1\right)
\left\langle\nabla w,\nabla g\right\rangle
\\&~~~~-\frac{2\mu_1|\nabla g|^2}{(\xi-g)^2}
-\frac{2\gamma_1\,|\nabla g|}{(\xi-g)^2}-\Upsilon_3,
\endaligned
\end{equation}
where
\begin{equation}\label{UPSLI3} \Upsilon_3:=\Upsilon_2 +
\frac{2\left\langle\nabla g,\nabla a+\partial_{u}a\nabla u\right\rangle\,|\nabla g|^2}{(\xi-g)^2}.\end{equation}

It is now convenient to factor out a term of the type
\begin{equation}\label{XIDEF9}
\Xi:=
2g'-\zeta n\left(\frac{g''}{g'}\right)^{2}(\xi-g)^2\end{equation} from the first three terms
in the right hand side of~\eqref{7tyfv6erdftwcgwuy2w} (up to a reminder).
For this, we write
\begin{equation}\label{760394hkS4}
\begin{split}
&\Bigg[2g'-\zeta n\left(\frac{g''}{g'}\right)^{2}(\xi-g)^2
\Bigg]\frac{a|D^2 g|^{2}}{(\xi-g)^2}
+\frac{2ag'\left\langle\nabla|\nabla g|^2,\nabla g\right\rangle}{(\xi-g)^3}
+\left(2g'-\frac{1}{\zeta} \right)\frac{a|\nabla g|^4}{(\xi-g)^4}
\\
=\;&\Xi\,\frac{a|D^2 g|^{2}}{(\xi-g)^2}
+\left(\Xi
+\zeta n\left(\frac{g''}{g'}\right)^{2}(\xi-g)^2
\right)\frac{a\left\langle\nabla|\nabla g|^2,\nabla g\right\rangle}{(\xi-g)^3}
\\&\qquad+\left(\Xi+\zeta n\left(\frac{g''}{g'}\right)^{2}(\xi-g)^2
-\frac{1}{\zeta} \right)\frac{a|\nabla g|^4}{(\xi-g)^4}\\
=\;& \Xi\,\left( \frac{a|D^2 g|^{2}}{(\xi-g)^2}+
\frac{a\left\langle\nabla|\nabla g|^2,\nabla g\right\rangle}{(\xi-g)^3}
+\frac{a|\nabla g|^4}{(\xi-g)^4}\right)
\\&\qquad+
\zeta n\left(\frac{g''}{g'}\right)^{2}
\frac{a\left\langle\nabla|\nabla g|^2,\nabla g\right\rangle}{\xi-g}
+\left(\zeta n\left(\frac{g''}{g'}\right)^{2}(\xi-g)^2
-\frac{1}{\zeta} \right)\frac{a|\nabla g|^4}{(\xi-g)^4}
.
\end{split}\end{equation}
Now we claim that
\begin{equation}\label{XIEPS}
\Xi\ge0.
\end{equation}
Indeed, recalling~\eqref{101}, \eqref{ZETA:SDE} and~\eqref{XIDEF9},
\begin{eqnarray*}
\Xi=2g'-
\frac{\sqrt{n}\,|g''|\,(\xi-g)}{g'}=
2F'(u)-\frac{\sqrt{n}\,|F''(u)|\,u\,(\xi-G(u))}{F'(u)}
\end{eqnarray*}
and therefore~\eqref{XIEPS} is a consequence of~\eqref{CONDFOD}.

We also remark that~$\nabla|\nabla g|^2=2D^2g\nabla g$, and consequently
\begin{eqnarray*}0&\le&\left(\frac{\left\langle\nabla|\nabla g|^2,\nabla g\right\rangle}{2(\xi-g)|\nabla g|^2}
+\frac{|\nabla g|^2}{(\xi-g)^{2}}\right)^2\\
&=&
\left(\frac{\left\langle\nabla|\nabla g|^2,\nabla g\right\rangle}{2(\xi-g)|\nabla g|^2}\right)^2
+
\frac{\left\langle\nabla|\nabla g|^2,\nabla g\right\rangle}{(\xi-g)^{3}}
+\frac{|\nabla g|^4}{(\xi-g)^4}\\&=&
\left(\frac{\left\langle D^2g\, \nabla g,\nabla g\right\rangle}{(\xi-g)|\nabla g|^2}\right)^2
+
\frac{\left\langle\nabla|\nabla g|^2,\nabla g\right\rangle}{(\xi-g)^{3}}
+\frac{|\nabla g|^4}{(\xi-g)^4}\\&\le&
\frac{|D^2g|^2}{(\xi-g)^2}+
\frac{\left\langle\nabla|\nabla g|^2,\nabla g\right\rangle}{(\xi-g)^3}
+\frac{|\nabla g|^4}{(\xi-g)^4}.
\end{eqnarray*}
This information, \eqref{760394hkS4} and~\eqref{XIEPS} give that
\begin{eqnarray*}
&&\Bigg[2g'-\zeta n\left(\frac{g''}{g'}\right)^{2}(\xi-g)^2
\Bigg]\frac{a|D^2 g|^{2}}{(\xi-g)^2}
+\frac{2ag'\left\langle\nabla|\nabla g|^2,\nabla g\right\rangle}{(\xi-g)^3}
+\left(2g'-\frac{1}{\zeta} \right)\frac{a|\nabla g|^4}{(\xi-g)^4}\\&&\qquad\ge
\zeta n\left(\frac{g''}{g'}\right)^{2}\frac{a\left\langle\nabla|\nabla g|^2,
\nabla g\right\rangle}{\xi-g}
+\left(\zeta n\left(\frac{g''}{g'}\right)^{2}(\xi-g)^2
-\frac{1}{\zeta} \right)\frac{a|\nabla g|^4}{(\xi-g)^4}.
\end{eqnarray*}

Now, we insert this inequality into~\eqref{7tyfv6erdftwcgwuy2w},
thus finding that
\begin{equation}\label{1.38cj}
\aligned
ag'\Delta w-w_t&\ge
\zeta n\left(\frac{g''}{g'}\right)^{2}
\frac{a\left\langle\nabla|\nabla g|^2,\nabla g\right\rangle}{\xi-g}
+\left(2(\xi-g)+\zeta n\left(\frac{g''}{g'}\right)^{2}(\xi-g)^2
-\frac{1}{\zeta} \right)\frac{a|\nabla g|^4}{(\xi-g)^4}
\\&\qquad
+2a\left(\frac{g'}{\xi-g}-1\right)
\left\langle\nabla w,\nabla g\right\rangle-\frac{2\mu_1|\nabla g|^2}{(\xi-g)^2}
-\frac{2\gamma_1\,|\nabla g|}{(\xi-g)^2}-\Upsilon_3.
\endaligned
\end{equation}
Moreover, it is convenient to exploit~\eqref{DD5}
once again and note that
\begin{eqnarray*}
\frac{\left\langle\nabla|\nabla g|^2,\nabla g\right\rangle}{(\xi-g)^{2}}
+\frac{2|\nabla g|^4}{(\xi-g)^{3}}
=\left\langle
\frac{\nabla|\nabla g|^2}{(\xi-g)^{2}}+2\frac{|\nabla g|^2\nabla g}{(\xi-g)^{3}},\nabla g\right\rangle
=\left\langle\nabla w,\nabla g\right\rangle
\end{eqnarray*}
and consequently
\begin{eqnarray*}&&
\zeta n\left(\frac{g''}{g'}\right)^{2}
\frac{a\left\langle\nabla|\nabla g|^2,\nabla g\right\rangle}{\xi-g}
+\left(2(\xi-g)+\zeta n\left(\frac{g''}{g'}\right)^{2}(\xi-g)^2
-\frac{1}{\zeta} \right)\frac{a|\nabla g|^4}{(\xi-g)^4}\\&=&
\zeta n a\left(\frac{g''}{g'}\right)^{2}(\xi-g)\left(
\left\langle\nabla w,\nabla g\right\rangle-
\frac{2|\nabla g|^4}{(\xi-g)^{3}}
\right)\\&&\qquad
+\left(2(\xi-g)+\zeta n\left(\frac{g''}{g'}\right)^{2}(\xi-g)^2
-\frac{1}{\zeta} \right)\frac{a|\nabla g|^4}{(\xi-g)^4}\\&=&
\zeta n a\left(\frac{g''}{g'}\right)^{2}(\xi-g)
\left\langle\nabla w,\nabla g\right\rangle
+\left(2(\xi-g)-\zeta n\left(\frac{g''}{g'}\right)^{2}(\xi-g)^2
-\frac{1}{\zeta} \right)\frac{a|\nabla g|^4}{(\xi-g)^4}
.\end{eqnarray*}
We can thereby plug this information into~\eqref{1.38cj}
and deduce that
\begin{equation}\label{DD13}
\aligned
ag'\Delta w-w_t&\ge \left(2(\xi-g)-\zeta n\left(\frac{g''}{g'}\right)^{2}(\xi-g)^2
-\frac{1}{\zeta} \right)\frac{a|\nabla g|^4}{(\xi-g)^4}
\\&\qquad
+a\left(\frac{2g'}{\xi-g}-2
+\zeta n \left(\frac{g''}{g'}\right)^{2}(\xi-g)
\right)
\left\langle\nabla w,\nabla g\right\rangle\\&\qquad-\frac{2\mu_1|\nabla g|^2}{(\xi-g)^2}
-\frac{2\gamma_1\,|\nabla g|}{(\xi-g)^2}-\Upsilon_3.
\endaligned
\end{equation}
Now we remark that
\begin{eqnarray*}
&&2(\xi-g)-\zeta n\left(\frac{g''}{g'}\right)^{2}(\xi-g)^2
-\frac{1}{\zeta}
=2(\xi-g)\left(1-\sqrt{n}\frac{|g''|}{g'}\right)
\\&&\qquad=2(\xi-g)\left(1-\sqrt{n}\frac{|F''(u)|\,u}{F'(u)}\right)\ge
2\kappa(\xi-g),
\end{eqnarray*}
thanks to~\eqref{kappapi}, \eqref{101}, \eqref{xig} and~\eqref{ZETA:SDE}.

For this reason, recalling the definition of~$w$ in~\eqref{DEFw},
we obtain that
$$ \left(2(\xi-g)-\zeta n\left(\frac{g''}{g'}\right)^{2}(\xi-g)^2
-\frac{1}{\zeta} \right)\frac{a|\nabla g|^4}{(\xi-g)^4}\ge
\frac{2\kappa a|\nabla g|^4}{(\xi-g)^3}=2a \kappa (\xi-g) w^2.
$$
Hence, by~\eqref{DD13},
\begin{equation}\label{ScjiengFUA}
\aligned
ag'\Delta w-w_t&\ge
2a \kappa (\xi-g) w^2
+a\left(\frac{2g'}{\xi-g}-2
+\zeta n \left(\frac{g''}{g'}\right)^{2}(\xi-g)
\right)
\left\langle\nabla w,\nabla g\right\rangle\\&\qquad-\frac{2\mu_1|\nabla g|^2}{(\xi-g)^2}
-\frac{2\gamma_1\,|\nabla g|}{(\xi-g)^2}-\Upsilon_3.
\endaligned
\end{equation}
We also note that
\begin{eqnarray*}
&&\frac{2g'}{\xi-g}-2
+\zeta n \left(\frac{g''}{g'}\right)^{2}(\xi-g)
=
\frac{2g'}{\xi-g}-2
+\sqrt{n }\frac{|g''|}{g'}=2\lambda,
\end{eqnarray*}
thanks to~\eqref{DEFlambda} and~\eqref{ZETA:SDE}.

Using this identity and~\eqref{DEFw}
inside~\eqref{ScjiengFUA}, we get that
\begin{equation}\label{1.52c}
\aligned
ag'\Delta w-w_t&\ge
2a \kappa (\xi-g) w^2
+2a\lambda
\left\langle\nabla w,\nabla g\right\rangle-2\mu_1 w
-\frac{2\gamma_1\,|\nabla g|}{(\xi-g)^2}-\Upsilon_3.
\endaligned
\end{equation}
Now we observe that, by~\eqref{UPSI1},
\eqref{UPSIL2} and~\eqref{UPSLI3},
\begin{equation}\label{76090756724}
\begin{split}
\Upsilon_3\,&=
\frac{2\left\langle \nabla g,\Upsilon_1\right\rangle}{(\xi-g)^2}
+\frac{2g'
\Delta g\,\left\langle \nabla g,\nabla a+\partial_{u}a\nabla u\right\rangle}{(\xi-g)^2}
+\frac{2\left\langle\nabla g,\nabla a+\partial_{u}a\nabla u\right\rangle\,|\nabla g|^2}{(\xi-g)^2}\\
&=
\frac{2}{(\xi-g)^2}\left\langle \nabla g,
\frac{g'(v)}{u}\left(\sum_{i=1}^n\partial_{\omega_i} H\nabla
\partial_{x_i} u+\sum_{i,j=1}^n\partial_{\Omega_{ij}} H
\nabla\partial_{x_i x_j}^2 u\right)\right\rangle\\&\qquad
+\frac{2g'
\Delta g\,\left\langle \nabla g,\nabla a+\partial_{u}a\nabla u\right\rangle}{(\xi-g)^2}
+\frac{2\left\langle\nabla g,\nabla a+\partial_{u}a\nabla u\right\rangle\,|\nabla g|^2}{(\xi-g)^2}
.\end{split}\end{equation}
From~\eqref{100} and~\eqref{BB}, we also note that
$$ \Delta g=g''|\nabla v|^2+g'\Delta v$$
and
$$ \Delta v={\rm div}\left(\frac{\nabla u}u\right)
=\frac{\Delta u}{u}-\frac{|\nabla u|^2}{u^2}=
\frac{\Delta u}{u}-|\nabla v|^2
.$$
These observations lead to
\begin{equation}\label{cancel}
\begin{split}
&g'\Delta g+|\nabla g|^2=g'
g''|\nabla v|^2+(g')^2\Delta v+(g')^2|\nabla v|^2
\\&\qquad=
g'g''|\nabla v|^2+(g')^2\left(\frac{\Delta u}{u}-|\nabla v|^2\right)+(g')^2|\nabla v|^2
\\&\qquad=
F'(u)\left(uF''(u)\frac{|\nabla u|^2}{u^2}+F'(u)\frac{\Delta u}{u}\right)
\\&\qquad=\frac{F'(u)\big({\rm div}(F'(u)\nabla u)\big)}{u}
.\end{split}\end{equation}
As a result, after an interesting cancellation we conclude that
\begin{equation}\label{part1}
\begin{split}
&\frac{2g'
\Delta g\,\left\langle \nabla g,\partial_{u}a\nabla u\right\rangle}{(\xi-g)^2}+
\frac{2\left\langle\nabla g,\partial_{u}a\nabla u\right\rangle\,|\nabla g|^2}{(\xi-g)^2}
\\&\quad
=\frac{2F'(u)\big({\rm div}(F'(u)\nabla u)\big)\partial_{u}a\left\langle\nabla g,\nabla u\right\rangle}{u(\xi-g)^2}
=
\frac{2\big({\rm div}(F'(u)\nabla u)\big)\partial_{u}a\left\langle\nabla g,
\nabla g\right\rangle}{(\xi-g)^2}\\&\quad
=2\partial_{u}a\big({\rm div}(F'(u)\nabla u)\big)w
\le 2\mu_2\,w,
\end{split}
\end{equation}
where the definition of~$\mu_2$ given in~\eqref{MUDUE}
has been used in the  inequality. Using again~\eqref{cancel},
\begin{equation}\label{part2}
\begin{split}
&\frac{2g'
\Delta g\,\left\langle \nabla g,\nabla a\right\rangle}{(\xi-g)^2}+
\frac{2\left\langle\nabla g,\nabla a\right\rangle\,|\nabla g|^2}{(\xi-g)^2}
\\&\qquad
=\frac{2F'(u)\big({\rm div}(F'(u)\nabla u)\big)\left\langle\nabla g,\nabla a\right\rangle}{u(\xi-g)^2}
\le\frac{2\gamma_2|\nabla g|}{(\xi-g)^2},
\end{split}
\end{equation}
where we have used the definition of~$\gamma_2$
in~\eqref{GAMMA2}.

{F}rom~\eqref{76090756724}, ~\eqref{part1} and~\eqref{part2},
we see that
\begin{equation}\label{76868583}
\begin{split}
\Upsilon_3\,&\le
\frac{2}{(\xi-g)^2}\left\langle \nabla g,
\frac{g'(v)}{u}\left(\sum_{i=1}^n\partial_{\omega_i} H\nabla
\partial_{x_i} u+\sum_{i,j=1}^n\partial_{\Omega_{ij}} H
\nabla\partial_{x_i x_j}^2 u\right)\right\rangle
+2\mu_2 w+\frac{2\gamma_2|\nabla g|}{(\xi-g)^2}.\end{split}\end{equation}

Now, recalling the definition of~$\gamma_3$
in~\eqref{GAMMA32}, we have that
\begin{eqnarray*}&&
\left|\frac{g'(v)}{u}\left(\sum_{i=1}^n\partial_{\omega_i} H\nabla
\partial_{x_i} u+\sum_{i,j=1}^n\partial_{\Omega_{ij}} H
\nabla\partial_{x_i x_j}^2 u\right)\right|\\&&\qquad\qquad
=
\left|\frac{F'(u)}{u}\left(\sum_{i=1}^n\partial_{\omega_i} H\nabla
\partial_{x_i} u+\sum_{i,j=1}^n\partial_{\Omega_{ij}} H
\nabla\partial_{x_i x_j}^2 u\right)\right|
\le\gamma_3.
\end{eqnarray*}
This and~\eqref{76868583} yield that
$$\Upsilon_3\le\frac{2(\gamma_2+\gamma_3)|\nabla g|}{(\xi-g)^2}+2\mu_2 w.$$
Combining this estimate and~\eqref{1.52c} we obtain the desired result in~\eqref{LU190h010}.
\end{proof}

\section{Cut-off functions and localization procedures}\label{KJSMD93sad45}

In order to obtain the global bounds in Theorem~\ref{D2},
we distinguish four
regimes, according to
the cut-off functions in~\eqref{161}.
For this, we recall the following auxiliary cut-off functions, both in the space and
in the time variables, that have been introduced in Lemmata~2.2 and~2.3
in~\cite{ZU2}:

\begin{lem}\label{PSI}
Let~$\theta\in(0,1)$, $R>0$ and $\rho \in (0,R)$. Then, there exists a decreasing
function~$\bar{\psi}\in C^2(\R,[0,1])$ such that
\begin{equation}\label{7132r3yhsjdsid}
{\mbox{$\bar{\psi}(r)=1$ for all~$r\in[0,R-\rho]$, \quad $\bar{\psi}(r)=0$
for all~$r\ge R$,}}
\end{equation}
and, for every~$r\ge0$,
\begin{equation}\label{si-a}
\rho|\bar{\psi}'(r)|+\rho^2|\bar{\psi}''(r)|\le C\big(\bar{\psi}(r)\big)^{\theta},
\end{equation}
for some~$C>0$, depending only
on~$\theta$.
\end{lem}

\begin{lem}\label{ilLECAS2M}
Let~$t_0\in\R$ and~$T>0$.
Let~$\theta\in(0,1)$ and~$\delta\in(0,T)$. Then, there exists an increasing
function~$\phi\in C^2(\R,[0,1])$ such that
\begin{equation}\label{7132r3yhsjdsid-t}
{\mbox{$\phi(t)=0$ for all~$t\le t_0-T$, and~$\phi(t)=1$
for all~$t\ge t_0-T+\delta$,}}
\end{equation}
and, for every~$t\in\R$,
\begin{equation}\label{si-a-t}
\delta|\phi'(t)|\le C\big(\phi(t)\big)^{\frac{1+\theta}2},
\end{equation}
for some~$C>0$, depending only
on~$\theta$.
\end{lem}

Now we obtain a general inequality for
the auxiliary barrier~$w$ introduced in~\eqref{DEFw}
in dependence of
a smooth and positive function~$\psi$:

\begin{lem}
Let~$\psi\in C^\infty(Q_{R,T}, \,(0,+\infty))$. Then, there exists~$C>0$, depending only on~$a_0$
and~$\kappa$, such that
\begin{equation}\label{DD15}
\aligned
&\qquad\frac{ag'\Delta (w\psi)-(w\psi)_t}{2}-a\left\langle\nabla(w\psi),\frac{g'\nabla\psi}{\psi}+\lambda\nabla g \right\rangle
\\&\ge\,\frac{a_{0}\kappa (\xi-g) w^2\psi}{4}
-a\lambda w
\left\langle \nabla\psi,\nabla g\right\rangle-\frac{C\mu^2\psi}{(\xi-g)}
-\frac{C\gamma^{4/3}\,\psi}{(\xi-g)^{5/3}}
\\&\qquad-\frac{ag'w|\nabla\psi|^{2}}{\psi}+\frac{(ag'\Delta\psi-\psi_{t})w}{2}.
\endaligned
\end{equation}
\end{lem}

\begin{proof} We have that
\begin{equation}\label{DD15ed}
\begin{split}&\frac{ag'\Delta(w\psi)-(w\psi)_t}{2}-\frac{ag'
\langle\nabla(w\psi),\nabla\psi\rangle
}{\psi}\\&\quad=\frac{(ag'\Delta w-w_t)\,\psi}{2}+\frac{(ag'\Delta\psi-\psi_t)\,w}2
-\frac{ag'w\,|\nabla\psi|^2}\psi.\end{split}
\end{equation}
Hence, subtracting~$a\lambda\langle\nabla(w\psi),\nabla g\rangle$
to both sides of~\eqref{DD15ed},
\begin{equation*}
\aligned
&\frac{ag'\Delta (w\psi)-(w\psi)_t}{2}-a\left\langle\nabla(w\psi),\frac{g'\nabla\psi}{\psi}+\lambda\nabla g \right\rangle
\\&=\frac{(ag'\Delta w-w_t)\,\psi}{2}+\frac{(ag'\Delta\psi-\psi_t)\,w}2
-\frac{ag'w\,|\nabla\psi|^2}\psi-a\lambda\langle\nabla(w\psi),\nabla g\rangle.
\endaligned
\end{equation*}
As a result, it follows from Lemma~\ref{Dl1} that
\begin{equation}\label{2.68bis}
\aligned
&\qquad\frac{ag'\Delta (w\psi)-(w\psi)_t}{2}-a\left\langle\nabla(w\psi),\frac{g'\nabla\psi}{\psi}+\lambda\nabla g \right\rangle
\\&\ge\,a \kappa (\xi-g) w^2\psi
+a\lambda
\left\langle\nabla w,\nabla g\right\rangle\psi-\mu w\psi
-\frac{\gamma\,|\nabla g|\psi}{(\xi-g)^2}
\\&\qquad-a\lambda\langle\nabla(w\psi),\nabla g\rangle-\frac{ag'w\,|\nabla\psi|^2}\psi+\frac{(ag'\Delta\psi-\psi_{t})w}{2}.
\endaligned
\end{equation}
One can also notice that
$$
\langle\nabla w,\nabla g\rangle\psi
-\langle\nabla(w\psi),\nabla g\rangle=-
w\,\langle\nabla \psi,\nabla g\rangle,
$$
which together with~\eqref{a0bound} and~\eqref{2.68bis} implies that
\begin{equation}\label{original0}
\aligned
&\qquad\frac{ag'\Delta (w\psi)-(w\psi)_t}{2}-a\left\langle\nabla(w\psi),\frac{g'\nabla\psi}{\psi}+\lambda\nabla g \right\rangle
\\&\ge\,a_{0}\kappa (\xi-g) w^2\psi
-a\lambda w
\left\langle \nabla\psi,\nabla g\right\rangle-\mu w\psi
-\frac{\gamma\,|\nabla g|\psi}{(\xi-g)^2}
\\&\qquad-\frac{ag'w|\nabla\psi|^{2}}{\psi}+\frac{(ag'\Delta\psi-\psi_{t})w}{2}.
\endaligned
\end{equation}

In addition, from~\eqref{DEFw} and Young's inequality with exponents~$4$ and~$4/3$,
\begin{equation}\label{DD25}\begin{split}
&\frac{\gamma\,|\nabla g|\psi}{(\xi-g)^2}=
\frac{\gamma\psi\,\sqrt{w}}{\xi-g}=\sqrt[4]{\xi-g}\;\sqrt{w}\;\sqrt[4]{\psi}\;
\frac{\gamma\psi^{\frac34}}{(\xi-g)^{\frac54}}
\\&\qquad\le
\frac{a_{0}\kappa(\xi-g) w^2\psi}{2}
+\frac{C\gamma^{4/3}\,\psi}{(a_{0}\kappa)^{1/3}(\xi-g)^{5/3}}
\\&\qquad\le
\frac{a_{0}\kappa(\xi-g) w^2\psi}{2}
+\frac{C\gamma^{4/3}\,\psi}{(\xi-g)^{5/3}},
\end{split}\end{equation}
for some~$C>0$, possibly varying from line to line
and possibly depending only on~$a_0$
and~$\kappa$. Formulas~\eqref{original0} and~\eqref{DD25} entail that
\begin{equation}\label{original1}
\aligned
&\qquad\frac{ag'\Delta (w\psi)-(w\psi)_t}{2}-a\left\langle\nabla(w\psi),\frac{g'\nabla\psi}{\psi}+\lambda\nabla g \right\rangle
\\&\ge\,\frac{a_{0}\kappa (\xi-g) w^2\psi}{2}
-a\lambda w
\left\langle \nabla\psi,\nabla g\right\rangle-\mu w\psi
-\frac{C\gamma^{4/3}\,\psi}{(\xi-g)^{5/3}}
\\&\qquad-\frac{ag'w|\nabla\psi|^{2}}{\psi}+\frac{(ag'\Delta\psi-\psi_{t})w}{2}.
\endaligned
\end{equation}
Besides, by the Cauchy-Schwarz inequality,
\begin{equation}\label{DD27}
\aligned
&\mu w\psi=\big(\sqrt{\xi-g}\;w\;\sqrt{\psi}\big)
\;\left(\frac{\mu\, \sqrt{\psi}}{\sqrt{\xi-g}}\right)\le
\frac{a_{0}\kappa(\xi-g) w^2\psi}{4}+
\frac{C\mu^2\psi}{a_{0}{\kappa(\xi-g)}}
\\&\qquad\le
\frac{a_{0}\kappa(\xi-g) w^2\psi}{4}+
\frac{C\mu^2\psi}{(\xi-g)},
\endaligned
\end{equation}
which combined with~\eqref{original1} gives the desired result in~\eqref{DD15}.
\end{proof}

Inequality~\eqref{DD15} will play a pivotal a role in the following computations
in order to analyze four different regimes, as given by~\eqref{161}.

\begin{lem}\label{Dl3}
In the setting of Theorem~\ref{D2}, if~$x\in B(x_0,R-\rho)$ and~$t\in [t_0-T,t_0]$,
\begin{equation}\label{CA:0}
w \le\Bigg[\tau_u^2+C\Big(\mu
+\gamma^{2/3}+\frac{
1}{\rho^2}+\frac{1}{\rho(R-\rho)}
+\frac{\sqrt{k_+ }}{\rho}\Big)\Bigg],
\end{equation}
for some~$C>0$, depending only on~$n$, $\eta$, $a_0$, $\kappa$ and~$\Gamma$.
Here, $\tau_u$, $\mu$ and~$\gamma$ are the quantities defined in~\eqref{sigtau},
\eqref{DEFMU22} and~\eqref{DEFMU33}, respectively.
\end{lem}

\begin{proof}
Let~$\theta\in(0,1)$, to be conveniently chosen in what follows.
For every~$x\in B(x_0,R)$,
we define \begin{equation}\label{PSIDEFINI}\psi(x):=\bar{\psi}(d(x,x_0)),\end{equation}
where~$d(\cdot,\cdot)$ represents the geodesic distance
and~$\bar{\psi}$ is the function introduced in  Lemma~\ref{PSI}.

Recalling
the assumption~\eqref{RIC} on the Ricci curvature, we have that
\begin{equation*}
\Delta d(x,x_0)\le\frac{n-1}{d(x,x_0)}+\sqrt{(n-1)k_+}.
\end{equation*}
As a result, in view of~\eqref{si-a}, we deduce that
\begin{equation}\label{DD16}
\begin{split}&
|\nabla\psi(x)|= |\bar{\psi}'(d(x,x_0))\,\nabla d(x,x_0)|\le
\frac{C\big( \psi(x)\big)^{\theta}}\rho\\{\mbox{and }}\quad&-\Delta\psi(x)=
-\bar{\psi}'(d(x,x_0))\Delta d(x,x_0)-
\bar{\psi}''(d(x,x_0))|\nabla d(x,x_0)|\\&\qquad\qquad\le
\frac{C\big( \psi(x)\big)^{\theta}}\rho\,\left(\frac{n-1}{d(x,x_0)}+\sqrt{(n-1)k_+}\right)+
\frac{C\big( \psi(x)\big)^{\theta}}{\rho^2},
\end{split}\end{equation}
with~$C>0$ depending only on~$\theta$.

We now define~$\widetilde w:=w\psi$ and,
in the support of~$\psi$, we exploit~\eqref{DD15}
and write that
\begin{equation}\label{DD17}
\aligned
&\qquad\frac{ag'\Delta \widetilde w-\widetilde w_t}{2}-a\left\langle\nabla\widetilde w,\frac{g'\nabla\psi}{\psi}+\lambda\nabla g \right\rangle
\\&\ge\,\frac{a_{0}\kappa (\xi-g) w^2\psi}{4}
-a\lambda w
\left\langle \nabla\psi,\nabla g\right\rangle-\frac{C\mu^2\psi}{(\xi-g)}
-\frac{C\gamma^{4/3}\,\psi}{(\xi-g)^{5/3}}
\\&\qquad-\frac{ag'w|\nabla\psi|^{2}}{\psi}+\frac{ag'w\Delta\psi}{2}.
\endaligned
\end{equation}
We take~$(x_1, t_1)$ in the closure of~$Q_{R,T}$ such that
\begin{equation}\label{M7}
\sup_{Q_{R,T}}\widetilde w=\widetilde w(x_1, t_1).\end{equation}
Since~$\widetilde w(x,t)=0$ if~$x\in\partial B(x_0,R)$,
necessarily~$x_1$ is an interior point of~$B(x_0,R)$.
Consequently~$\nabla\widetilde w(x_1,t_1)=0$ and~$\Delta \widetilde w(x_1,t_1)\le0$.
Hence, inserting this information into~\eqref{DD17},
we obtain that
\begin{equation}\label{DD18}
\begin{split}&
0 \ge\,\frac{\widetilde w_t}{2}+\frac{a_{0}\kappa (\xi-g) w^2\psi}{4}
-a\lambda w
\left\langle \nabla\psi,\nabla g\right\rangle-\frac{C\mu^2\psi}{(\xi-g)}
-\frac{C\gamma^{4/3}\,\psi}{(\xi-g)^{5/3}}
\\&\qquad-\frac{ag'w|\nabla\psi|^{2}}{\psi}+\frac{ag'w\Delta\psi}{2}\Bigg|_{(x,t)=(x_1,t_1)}.
\end{split}
\end{equation}
Exploiting~\eqref{DEFg}, \eqref{101}, \eqref{DEFv} and~\eqref{DEFw},
we also see that
\begin{equation}\label{eaS}\begin{split}&
\frac{|F'(u(x,t))|^{2}|\nabla u(x,t)|^{2}}{|u(x,t)|^{2}(\xi-G(u(x,t)))^{2}}=
\frac{|g'(v(x,t))|^{2}|\nabla u(x,t)|^{2}}{|u(x,t)|^{2}(\xi-g(v(x,t)))^{2}}\\&\qquad=
\frac{|g'(v(x,t))|^{2}|\nabla v(x,t)|^{2}}{ (\xi-g(v(x,t)))^{2}}=
w(x,t).\end{split}
\end{equation}

Now, to address the proof of~\eqref{CA:0},
it is convenient to distinguish two cases, namely:
\begin{eqnarray}
\label{CA:1}&&{\mbox{either $t_1=t_0-T$,}}
\\&&{\mbox{or $t_1\in (t_0-T,t_0]$.}}\label{CA:2}
\end{eqnarray}
To start with, we suppose that~\eqref{CA:1} holds true.
In this case, we use~\eqref{sigtau},
\eqref{M7}
and~\eqref{eaS} to deduce that,
for every~$(x,t)\in Q_{R,T}$,
\begin{equation*}
\aligned
\widetilde w (x,t)&\leq \widetilde w(x_1,t_0-T)\\&\leq
\sup_{x\in B(x_0,R)} \widetilde w(x,t_0-T)\\&\leq\sup_{x\in B(x_0,R)}w(x,t_0-T)\\&=\sup_{x\in B(x_0,R)}\frac{\big(F'(u)\big)^{2}|\nabla u|^2}{u^2(\xi-G(u))^2}(x,t_0-T)
\\&\leq \tau_u^2.
\endaligned
\end{equation*}
In particular, for all~$(x,t)\in B(x_0,R-\rho) \times[t_0-T,t_0]$,
$$ w(x,t)=\widetilde w (x,t)\le\tau_u^2,$$
and this proves~\eqref{CA:0} in this case.

Hence, to complete the proof of~\eqref{CA:0}, we now consider the
case in which~\eqref{CA:2} is satisfied. Then, $\widetilde w_t(x_1,t_1)\ge0$,
and consequently~\eqref{DD18} entails that
\begin{equation}\label{DD19}
\begin{split}&
0 \ge\,\frac{a_{0}\kappa (\xi-g) w^2\psi}{4}
-a\lambda w
\left\langle \nabla\psi,\nabla g\right\rangle -\frac{C\mu^2\psi}{(\xi-g)}
-\frac{C\gamma^{4/3}\,\psi}{(\xi-g)^{5/3}}
\\&\qquad-\frac{ag'w|\nabla\psi|^{2}}{\psi}+\frac{ag'w\Delta\psi}{2}\Bigg|_{(x,t)=(x_1,t_1)}.
\end{split}
\end{equation}
It is also useful to observe that
$$ \frac{\sqrt{n}\,|g''(r)|}{g'(r)}=\frac{\sqrt{n}\,e^r\,|F''(e^r)|}{F'(e^r)}\le\frac{2F'(e^r)}{\xi-G(e^r)}=
\frac{2g'(r)}{\xi-g(r)},$$
thanks to \eqref{101} and~\eqref{CONDFOD}.

{F}rom this and the definition of~$\lambda$ in~\eqref{DEFlambda},
we conclude that
$$|\lambda(r)|\le\frac{2g'(r)}{\xi-g(r)}+1.$$
It is also useful to observe that, in light of \eqref{SUPER}, \eqref{DEFg}
and~\eqref{101},
\begin{equation}\label{Gamma}
\frac{2g'(r)}{\xi-g(r)}=\frac{2F'(e^r)}{\xi-G(e^r)}\le2\Gamma,
\end{equation}
and thus~$|\lambda(r)|\le2\Gamma+1$.

As a consequence, recalling the definition of~$w$ in~\eqref{DEFw},
and utilizing the
Young's inequality with exponents~$4/3$ and~$4$,
we see that, in the support of~$\psi$,
\begin{equation*}\begin{split}& a\left|
\lambda\,w\left\langle\nabla \psi,\nabla g\right\rangle\right|\le(2\Gamma+1)
a_{0}^{-1}w\,|\nabla \psi|\;|\nabla g|\\&\qquad=
(2\Gamma+1)
a_{0}^{-1}w^{3/2}\,|\nabla \psi|\;(\xi-g)
\\&\qquad=(2\Gamma+1)
\left(\frac{a_{0}^{3/4}\,[\kappa(\xi-g)]^{3/4}\,w^{3/2}\,\psi^{3/4}}{3}\right)
\;\frac{3a_{0}^{-7/4}(\xi-g)^{1/4}\,|\nabla \psi|}{ \kappa^{3/4}\psi^{3/4}}
\\&\qquad\le \frac18a_{0}\kappa\,(\xi-g)\,w^{2}\,\psi+\frac{C\;(\xi-g)\,|\nabla \psi|^{4}}{\psi^{3}}
\end{split}
\end{equation*}
for some constant~$C>0$
depending only on~$a_0$, $\kappa$ and~$\Gamma$.
{F}rom this and~\eqref{DD16}, we find that
\begin{equation}\label{DD20}
a \left|
\lambda\,w\left\langle\nabla \psi,\nabla g\right\rangle\right|\le
\frac18a_{0}\,(\xi-g)\kappa\,w^{2}\,\psi+\frac{C\;(\xi-g)\, \psi^{4\theta-3}}{\rho^{4}},\end{equation}
up to renaming~$C>0$, possibly also in dependence
of~$\theta$.

Also, by using again~\eqref{101}, \eqref{DD16}, \eqref{Gamma}
and the Cauchy-Schwarz inequality, we can write that
\begin{equation}\label{DD23}
\begin{split}&\frac{ag'w\,|\nabla\psi|^2}\psi\le
\frac{Ca^{-1}_{0}g'w\psi^{2\theta-1}}{\rho^2}=
a^{-1}_{0}g'\sqrt{\xi-g}\;w\;\sqrt{\psi}\;
\frac{C\psi^{2\theta-\frac32}}{\sqrt{\xi-g}\;\rho^2}
\\&\qquad\qquad\qquad\le\frac{1}{16}a_{0}
(\xi-g)\kappa w^2\psi+\frac{
C(g')^{2}\psi^{4\theta-3}}{a^{3}_{0}(\xi-g)\kappa\rho^4}
\\&\qquad\qquad\qquad\le\frac{1}{16}a_{0}
(\xi-g)\kappa w^2\psi+\frac{
C(\xi-g)\psi^{4\theta-3}}{\rho^4},
\end{split}\end{equation}
up to renaming~$C$.

Plugging this information and~\eqref{DD20} into~\eqref{DD19},
we obtain that, at the point~$(x_1,t_1)$,
\begin{equation}\label{DD21}
\frac{a_{0}\kappa (\xi-g) w^2\psi}{16}\leq\,\frac{C\mu^2\psi}{(\xi-g)}+\frac{C\gamma^{4/3}\,\psi}{(\xi-g)^{5/3}}
+\frac{C(\xi-g)\, \psi^{4\theta-3}}{\rho^{4}}
-\frac{ag'w\Delta\psi}{2}
.
\end{equation}
Moreover, by~\eqref{DD16},
using the short notation~$d:=d(x,x_0)$, we see that
\begin{eqnarray*}
-\frac{\Delta\psi\,w}2&\le&\frac{Cw}{2}\,\left[
\frac{\psi^{\theta}}{ \rho}\,\left(\frac{n-1}{d}+\sqrt{(n-1)k_+}\right)+
\frac{\psi^{\theta} }{ \rho^2}\right]\\
&=&
\frac{C\,\sqrt{\xi-g}\;w\;\sqrt{\psi}}{2}\,\left[
\rho\,\left(\frac{n-1}{d}+\sqrt{(n-1)k_+}\right)+1
\right]\frac{\psi^{{\theta}-\frac12} }{ \sqrt{\xi-g}\;\rho^2}.
\end{eqnarray*}
Consequently, by the Cauchy-Schwarz inequality and~\eqref{Gamma},
\begin{equation}\label{Qua4d}
\aligned
-\frac{ag'\Delta\psi\,w}2
 &\le\frac{a_{0}\kappa(\xi-g) w^2\psi}{32}+C\left[
 \rho\,\left(\frac{n-1}{d}+\sqrt{(n-1)k_+}\right)+1
 \right]^2\frac{a^{-3}_{0}(g')^{2}\psi^{2\theta-1} }{ (\xi-g)\kappa\rho^4}\\
&\leq\frac{a_{0}\kappa(\xi-g) w^2\psi}{32}+C\left[
\rho\,\left(\frac{n-1}{d}+\sqrt{(n-1)k_+}\right)+1
\right]^2\frac{\psi^{2\theta-1} (\xi-g)}{\rho^4}
\endaligned
,\end{equation}
up to renaming~$C$.

We also remark that when~$x\in B(x_0,R-\rho)$, we have that~$d=d(x,x_0)\in[0,R-\rho)$,
and thus~$\frac{R-d}{\rho}>1$, which gives that~$
\psi(x)=\bar{\psi}(d)=\alpha\left(\frac{R-d}{\rho}\right)=1$.
In particular,
\begin{equation} \label{0987y0iuj0098j}\Delta\psi(x)=0\qquad{\mbox{ for all }}x\in B(x_0,R-\rho).\end{equation}
Notice also that
\begin{equation}\label{last-A}
\aligned
-\frac{ag'\Delta\psi\,w}2&\le
\frac{a_{0}\kappa(\xi-g) w^2\psi}{32}+C\left[
\rho\,\left(\frac{n-1}{R-\rho}+\sqrt{(n-1)k_+}\right)+1
\right]^2\frac{\psi^{2\theta-1} (\xi-g) }{\rho^4}
\\&\le
\frac{a_{0}\kappa(\xi-g) w^2\psi}{32}+
\frac{C\psi^{2\theta-1} (\xi-g) }{\rho^2(R-\rho)^2}
+\frac{C k_+\psi^{2\theta-1} (\xi-g) }{\rho^2}+
\frac{C\psi^{2\theta-1}(\xi-g) }{\rho^4},
\endaligned
\end{equation}
up to renaming~$C$. Indeed, the estimate in~\eqref{last-A}
is obvious in~$B(x_0,R-\rho)$,
since the right hand side vanishes, thanks to~\eqref{0987y0iuj0098j},
and it follows directly
from~\eqref{Qua4d}
in the complement of~$B(x_0,R-\rho)$,
where~$d\ge R-\rho$.

We can therefore insert~\eqref{last-A} into~\eqref{DD21}
and find that, at the point~$(x_1,t_1)$,
\begin{equation*}
\aligned
\frac{a_{0}\kappa (\xi-g) w^2\psi}{32}\leq\,&\frac{C\mu^2\psi}{(\xi-g)}+\frac{C\gamma^{4/3}\,\psi}{(\xi-g)^{5/3}}
+\frac{
C\psi^{4\theta-3}(\xi-g)}{\rho^4}
\\&\qquad+\frac{C\psi^{2\theta-1}(\xi-g) }{\rho^2(R-\rho)^2}
+\frac{C k_+\psi^{2\theta-1}(\xi-g)}{\rho^2}+
\frac{C\psi^{2\theta-1} (\xi-g) }{\rho^4}.
\endaligned
\end{equation*}
In light of~\eqref{M7},
we can rewrite the latter estimate as
\begin{equation}\label{DD29}\begin{split}
\sup_{Q_{R,T}}w^2 \psi^2 &\le
\frac{C\mu^2\psi^{2}}{(\xi-g)^{2}}+\frac{C\gamma^{4/3}\,\psi^{2}}{(\xi-g)^{8/3}}
+\frac{
C\psi^{4\theta-2}}{\rho^4}
+\frac{C\psi^{2\theta} }{\rho^2(R-\rho)^2}
+\frac{C k_+\psi^{2\theta} }{\rho^2}+
\frac{C\psi^{2\theta} }{\rho^4}.\end{split}
\end{equation}
We recall that~$0\le\psi\le1$ and that~$\psi=1$
in~$B(x_0,R-\rho)$. In this way, choosing $\theta:=1/2$,
we deduce from~\eqref{DD29} and~\eqref{XIPOS} that
\begin{equation*}\begin{split}&\sup_{B(x_0,R-\rho)
\times(t_0-T,t_0]
}w^2 =
\sup_{B(x_0,R-\rho)
\times(t_0-T,t_0]
}w^2 \psi^2 \\&\qquad\qquad\le
C\mu^2
+C\gamma^{4/3}+\frac{
C}{\rho^4}+\frac{C }{\rho^2(R-\rho)^2}
+\frac{C k_+ }{\rho^2}
,\end{split}
\end{equation*}
where~$C$ can now depend on~$\eta$ too.
With this, the proof of~\eqref{CA:0} is complete.\end{proof}

\begin{lem}\label{Dl4}
In the setting of Theorem~\ref{D2}, if~$x\in B(x_0,R)$ and~$t\in [t_0-T+\delta,t_0]$,
\begin{equation}\label{C3456dffA:0}\begin{split}
&w(x,t)\le\Bigg[\sigma_u^2+C\bigg(\mu+\gamma^{2/3}
+\frac{1}{\delta}\bigg)\Bigg],\end{split}
\end{equation}
for some~$C>0$, depending only on~$\eta$, $a_0$ and $\kappa$.
Here, $\sigma_u$, $\mu$ and~$\gamma$ are the quantities defined in~\eqref{sigtau},
\eqref{DEFMU22} and~\eqref{DEFMU33}, respectively.
\end{lem}

\begin{proof} We take~$\phi$ as in Lemma~\ref{ilLECAS2M} (say, with~$\theta :=1/2$), and we define~$\widetilde w(x,t):=w(x,t)\phi(t)$. Then, in light of~\eqref{DD15},
\begin{equation}\label{DD31}
\frac{ag'\Delta \widetilde w-\widetilde w_t}{2}-a\left\langle\nabla\widetilde w,\lambda\nabla g \right\rangle
\ge\,\frac{a_{0}\kappa (\xi-g) w^2\phi}{4}
-\frac{C\mu^2\phi}{(\xi-g)}
-\frac{C\gamma^{4/3}\,\phi}{(\xi-g)^{5/3}}-\frac{\phi_{t}w}{2}.
\end{equation}
Suppose that the maximum of $\widetilde w$ in the closure of~$Q_{R,T}$ is reached at $(x_1, t_1)$.
Since~$\widetilde w=0$ when~$t=t_0-T$, we know that~$t_1\in(t_0-T,t_0]$.
As a result,
\begin{equation}\label{DD32}
\widetilde w_t(x_1,t_1)\ge0.
\end{equation}
We then distinguish two cases,
\begin{eqnarray}
&&\label{SPALS:1}{\mbox{either }}x_1\in\partial B(x_0,R),\\
&&\label{SPALS:2}{\mbox{or }}
x_1\in B(x_0,R).
\end{eqnarray}
If~\eqref{SPALS:1} holds true, then, in~$ Q_{R,T}$,
\begin{equation*}
\widetilde w\leq \widetilde w(x_1,t_1)\leq
\sup_{{x\in\partial B(x_0,R)}\atop{t\in[t_0-T,t_0]}}\widetilde w(x,t)\leq\sup_{{x\in\partial B(x_0,R)}\atop{t\in[t_0-T,t_0]}}w(x,t).
\end{equation*}
Consequently, recalling the definition of~$w$ in~\eqref{DEFw} and using~\eqref{100} and~\eqref{BB},
we have that
\begin{equation*}
\widetilde w\le\sup_{{x\in\partial B(x_0,R)}\atop{t\in[t_0-T+\delta/2,t_0]}}\frac{\big(F'(u)\big)^{2}|\nabla u|^2}{u^2(\xi-G(u))^2}(x,t)=\sigma_u^2.
\end{equation*}
As a result, since~$\phi=1$ if~$t\ge t_0-T+\delta$, thanks to~\eqref{7132r3yhsjdsid-t},
we obtain that, if~$x\in B(x_0,R)$ and~$t\in[t_0-T+\delta, t_0]$,
$$ w(x,t)=\widetilde w(x,t)\le\sigma_u^2,$$
which proves~\eqref{C3456dffA:0} in this case.

Hence, we can now suppose that~\eqref{SPALS:2}
holds true. In this case, we have that~$\Delta \widetilde w(x_1,t_1)\leq0$ and~$\nabla \widetilde w(x_1,t_1)=0$.
Therefore, in the light of~\eqref{101}, \eqref{DD15} and~\eqref{DD32},
at the point~$(x_1,t_1)$ it holds that
\begin{equation}\label{DD33}
0\ge\,\frac{a_{0}\kappa (\xi-g) w^2\phi}{4}
-\frac{C\mu^2\phi}{(\xi-g)}
-\frac{C\gamma^{4/3}\,\phi}{(\xi-g)^{5/3}}-\frac{\phi_{t}w}{2}.
\end{equation}
Moreover, from~\eqref{si-a-t} and the Cauchy-Schwarz inequality,
\begin{equation}\label{DD35}
\begin{split}
 &\frac{\phi_t\,w}2\le
\frac{C\,\phi^{\frac{1+\theta}2}\,w}{2\delta}
=\left(\frac{\sqrt{a_0 \kappa (\xi-g)}\;w\;\sqrt{\phi}}2\right)\;
\left(\frac{C\phi^{\frac{\theta}2}}{\delta\sqrt{a_0 \kappa (\xi-g)}}\right)
\\&\qquad\le
\frac{a_{0}\kappa(\xi-g) w^2\phi}8+
\frac{C\phi^{\theta}}{\delta^2a_{0}\kappa(\xi-g)}\le
\frac{a_{0}\kappa(\xi-g) w^2\phi}8+
\frac{C\phi^{\theta}}{\delta^2(\xi-g)}.
\end{split}
\end{equation}
Plugging~\eqref{DD35} into~\eqref{DD33}, we conclude that, at the point~$(x_1,t_1)$,
$$\frac18\,
(\xi-g)a_{0}\kappa\,w^2 \phi\le
\frac{C\mu^2\phi}{{(\xi-g)}}+\frac{C\gamma^{4/3}\,\phi}{(\xi-g)^{5/3}}+\frac{C\phi^{\theta}}{\delta^2(\xi-g)} .$$
That is, at the point~$(x_1,t_1)$,
$$
w^2 \phi\le \frac{C\mu^2\phi}{{(\xi-g)^{2}}}+
\frac{C\gamma^{4/3}\,\phi}{(\xi-g)^{8/3}}+\frac{C\phi^{\theta}}{\delta^2(\xi-g)^{2}}.$$
Now, since~$0\le\phi\le1$ and~$\phi=1$ for any~$t\ge t_0-T+\delta$, this
implies that
$$ \sup_{B(x_0,R)\times [t_0-T+\delta,t_0]}w^2=
\sup_{B(x_0,R)\times [t_0-T+\delta,t_0]}w^2\phi^2\le
\frac{C\mu^2}{(\xi-g)^2}+\frac{C}{\delta^2(\xi-g)^2}+
\frac{C\gamma^{4/3}}{(\xi-g)^{8/3}}
\Bigg|_{(x,t)=(x_1,t_1)}.
$$
As a consequence, recalling also~\eqref{XIPOS},
we obtain~\eqref{C3456dffA:0}, as desired.
\end{proof}

\begin{lem}\label{Dl5}
In the setting of Theorem~\ref{D2}, if~$x\in B(x_0,R)$ and~$t\in [t_0-T,t_0]$,
\begin{equation}\label{C3456543665888885yhdffA:0}
w(x,t)\le[\sigma_u^2+\tau_u^2+C(\mu+\gamma^{2/3})],
\end{equation}
for some~$C>0$, depending only on~$\eta$, $a_0$ and~$\kappa$.
Here, $\tau_u$ and~$\sigma_u$
are the quantities defined in~\eqref{sigtau},
$\mu$
is defined in~\eqref{DEFMU22}, and~$\gamma$ is defined
in~\eqref{DEFMU33}.
\end{lem}

\begin{proof}
We suppose that the maximum of $w$ in the closure of~$Q_{R,T}$ is reached at the point~$(x_1, t_1)$.  We distinguish three possibilities:
\begin{eqnarray}
\label{CA:3-001} &&{\mbox{either }} x_1\in B(x_0,R) {\mbox{ and }}t_1\in(t_0-T,t_0]
,\\
\label{CA:3-002} &&{\mbox{or }} x_1\in B(x_0,R) {\mbox{ and }}t_1=t_0-T
,\\
\label{CA:3-003} &&{\mbox{or }} x_1\in \partial B(x_0,R) {\mbox{ and }}t_1\in[t_0-T,t_0].
\end{eqnarray}
Suppose first that~\eqref{CA:3-001} holds true. Then, we have that~$\Delta w(x_1,t_1)\le0$,
$\nabla w(x_1,t_1)=0$ and~$w_t(x_1,t_1)\ge0$. Therefore, in light of Lemma~\ref{Dl1}
and recalling also~\eqref{101},
we obtain that, at the point~$(x_1,t_1)$,
\begin{equation}\label{DD36}
0\ge a \kappa (\xi-g) w^2-\mu w-\frac{\gamma\,|\nabla g|}{(\xi-g)^2}
\ge a_0 \kappa (\xi-g) w^2-\mu w-\frac{\gamma\,|\nabla g|}{(\xi-g)^2}.
\end{equation}
We insert~\eqref{DD25} and~\eqref{DD27}
(used here with~$\psi:=1$)
into~\eqref{DD36} to see that, at the point~$(x_1,t_1)$,
\begin{equation*}
\frac{a_{0}\kappa (\xi-g) w^2}{4}\leq
\frac{C\mu^2}{(\xi-g)}
+\frac{C\gamma^{4/3}}{(\xi-g)^{5/3}}.
\end{equation*}
Consequently, using the maximality of~$(x_1,t_1)$ and recalling
that~$\xi-g\geq\eta>0$,
$$ \sup_{B(x_0,R)\times [t_0-T,t_0]}w^2\le
\frac{C\mu^2}{{(\xi-g)^{2}}}
+\frac{C\gamma^{4/3}}{(\xi-g)^{8/3}}
\Bigg|_{(x,t)=(x_1,t_1)}
\le C\mu^2+C\gamma^{4/3}.$$
This proves~\eqref{C3456543665888885yhdffA:0} in this case.
Thus, we can now assume that~\eqref{CA:3-002} is satisfied. Then, recalling~\eqref{DEFw},
\eqref{100} and~\eqref{BB},
we see that,
in~$Q_{R,T}$,
\begin{equation*} \begin{split}&
w\le w(x_1,t_0-T)=\frac{\big(F'(u)\big)^{2}|\nabla v|^2}{\xi-G(u)^2}(x_1,t_0-T)
\\&\qquad\qquad=
\frac{\big(F'(u)\big)^{2}|\nabla u|^2}{u^2(\xi-G(u))^2}(x_1,t_0-T)\le\tau_u^2,
\end{split}\end{equation*}
which establishes~\eqref{C3456543665888885yhdffA:0} in this case.

We now suppose that~\eqref{CA:3-003} is satisfied.
In such a case, we have that,
in~$Q_{R,T}$,
$$
w\le w(x_1,t_1)=\frac{\big(F'(u)\big)^{2}|\nabla u|^2}{u^2(\xi-G(u))^2}(x_1,t_1)\le\sigma_u^2,$$
whence the proof of~\eqref{C3456543665888885yhdffA:0} is complete.
\end{proof}

\begin{lem}\label{Dl6}
In the setting of Theorem~\ref{D2}, if~$x\in B(x_0,R-\rho)$ and~$t\in [t_0-T+\delta,t_0]$,
\begin{equation}\label{C3123456757374254777456dffA:0}
w(x,t)\le C\,\left(\mu+\gamma^{2/3}+ \frac{1}{\rho(R-\rho)}+
\frac{\sqrt{k_+}}{\rho}
+\frac{1}{\delta}+\frac{1}{\rho^2}
\right),
\end{equation}
for some~$C>0$, depending only on~$\eta$, $a_0$, $\kappa$ and~$\Gamma$.
Here, $\mu$ and~$\gamma$ are the quantities defined in~\eqref{DEFMU22} and~\eqref{DEFMU33}, respectively.
\end{lem}

\begin{proof} Let~$\theta\in(0,1)$ to be conveniently chosen in what follows.
Let also~$\psi$ be as in~\eqref{PSIDEFINI}
and~$\phi$ be as in Lemma~\ref{ilLECAS2M}. We define~$\Phi(x,t):=\psi(x)\phi(t)$ and~$\widetilde w:=w\Phi$.
Suppose that the maximum of $\widetilde w$ in $Q_{R,T}$ is reached at some point~$(x_1, t_1)$.
Since~$\Phi$ vanishes along the parabolic boundary, we know that~$x_1\in B(x_0,R)$
and~$t_1\in(t_0-T,t_0]$. As a consequence,
$$ \Delta\widetilde w(x_1,t_1)\le0,\qquad
\nabla\widetilde w(x_1,t_1)=0\qquad{\mbox{and}}\qquad\widetilde w_t(x_1,t_1)\ge0.$$
Combining this information with~\eqref{DD15}, we obtain that, at the point~$(x_1,t_1)$,
\begin{equation}\label{DD40}
\aligned
&0\ge\,\frac{a_{0}\kappa (\xi-g) w^2\Phi}{4}
-a\lambda w
\left\langle \nabla\Phi,\nabla g\right\rangle-\frac{C\mu^2\Phi}{(\xi-g)}
-\frac{C\gamma^{4/3}\,\Phi}{(\xi-g)^{5/3}}
\\&\qquad-\frac{ag'w|\nabla\Phi|^{2}}{\Phi}+\frac{(ag'\Delta\Phi-\Phi_{t})w}{2}\Bigg|_{(x,t)=(x_1,t_1)}.
\endaligned
\end{equation}
From~\eqref{DD20}, \eqref{DD23} and~\eqref{DD40}, we deduce that
\begin{equation}\label{part4A}
\aligned
&0\geq\,\frac{a_{0}\kappa (\xi-g) w^2\Phi}{16}-\frac{C\mu^2\Phi}{(\xi-g)}-\frac{C\gamma^{4/3}\,\Phi}{(\xi-g)^{5/3}}
-\frac{
C\Phi^{4\theta-3}(\xi-g)}{\rho^4}+\frac{(ag'\Delta\Phi-\Phi_{t})w}{2}\Bigg|_{(x,t)=(x_1,t_1)}.
\endaligned
\end{equation}
Now, from~\eqref{last-A},
\begin{equation}\label{part4-1}
\aligned
&\frac{-ag'w\Delta\Phi}{2}=\frac{-ag'w\phi\Delta\psi}{2}\\&\qquad\le
\frac{a_{0}\kappa(\xi-g) w^2\psi\phi}{32}+
\frac{C\psi^{2\theta-1}\phi (\xi-g)}{\rho^2(R-\rho)^2}
+\frac{C k_+\psi^{2\theta-1}\phi(\xi-g) }{ \rho^2}+
\frac{C\psi^{2\theta-1} \phi (\xi-g)}{\rho^4},
\endaligned
\end{equation}
and from~\eqref{DD35},
\begin{equation}\label{part4-2}
\frac{w\Phi_{t}}{2}=\frac{w\psi\phi_{t}}{2}\le
\frac{a_{0}\kappa(\xi-g) w^2\phi}{64}+
\frac{C\phi^{\theta}\psi}{\delta^2(\xi-g)}.
\end{equation}
{F}rom~\eqref{part4A}, \eqref{part4-1} and~\eqref{part4-2} we obtain that,
at the point~$(x_1,t_1)$,
\begin{eqnarray*}
\frac{a_0 \kappa (\xi-g)\,w^2 \Phi}{64} &\le&\frac{C\mu^2\Phi}{(\xi-g)}+\frac{C\gamma^{4/3}\,\Phi}{(\xi-g)^{5/3}}
+\frac{C\Phi^{4\theta-3}(\xi-g)}{\rho^4}+\frac{C\psi^{2\theta-1}\phi(\xi-g)}{\rho^2(R-\rho)^2}
\\&&+\frac{C k_+\psi^{2\theta-1}\phi (\xi-g)}{\rho^2}+
\frac{C\psi^{2\theta-1} \phi (\xi-g)}{\rho^4}+\frac{C\phi^{\theta}\psi}{\delta^2(\xi-g)}
.\end{eqnarray*}
We see that~$0\leq\Phi\leq1$, and that~$\Phi=1$
for every~$x\in B(x_0,R-\rho)$ and~$t\in [t_0-T+\delta,t_0]$. Thus, if~$x\in B(x_0,R-\rho)$
and~$t\in[t_0-T+\delta,t_0]$, choosing $\theta:=3/4$,
\begin{eqnarray*}
&&w^2(x,t)= w^2(x,t)\Phi^2(x,t)=\widetilde w^2(x,t)\le
\widetilde w^2(x_1,t_1)=
w^2(x_1,t_1)\Phi^2(x_1,t_1)\\&&
\qquad\le\frac{C\mu^2}{(\xi-g)^{2}}+\frac{C\gamma^{4/3}}{(\xi-g)^{8/3}}
+\frac{C}{\rho^{4}}+\frac{C}{\rho^2(R-\rho)^2}
+\frac{C k_+}{\rho^2}+\frac{C}{\delta^2(\xi-g)^{2}}\Bigg|_{(x,t)=(x_1,t_1)},
\end{eqnarray*}
that, recalling~\eqref{XIPOS}, yields the desired estimate in~\eqref{C3123456757374254777456dffA:0}.
\end{proof}

\section{Completion of the proof of Theorem~\ref{D2}}\label{KJS-M-D93sad45}

In this section, we provide the proof of Theorem~\ref{D2}.
To this end, we use the notation
\begin{equation}\label{TUTTITE}
\begin{split}
&{\widetilde{\mathscr{C}}}:=\mu+\gamma^{2/3},\\
&{\widetilde{\mathscr{T}}}:=\frac{1}{\delta}\\{\mbox{and }}\quad
&{\widetilde{\mathscr{S}}}:=\frac{1}{\rho^{2}}+\frac{1}{\rho(R-\rho)}+\frac{\sqrt{k_+}}{\rho}.
\end{split}
\end{equation}
With this notation, gathering together the estimates in Lemmata~\ref{Dl3},
\ref{Dl4}, \ref{Dl5} and~\ref{Dl6}, we obtain the following statement.

\begin{cor}\label{E67QRqwer3456788gfnh}
In the setting of Theorem~\ref{D2}, the function~$w$ can be
estimated by
\begin{eqnarray*}
C{\widetilde{\mathscr{C}}}+
\big(\tau_u^2+C{\widetilde{\mathscr{S}}}\big)&&
{\mbox{in~$ B(x_0,R-\rho)\times [t_0-T,t_0]$,}}\\C{\widetilde{\mathscr{C}}}+
\big(\sigma_u^2+ C{\widetilde{\mathscr{T}}}
\big)&&
{\mbox{in~$ B(x_0,R)\times [t_0-T+\delta,t_0]$,}}\\C{\widetilde{\mathscr{C}}}+
\big(\sigma_u^2+\tau_u^2
\big)&&
{\mbox{in~$B(x_0,R)\times[t_0-T,t_0]$,}}
\\C{\widetilde{\mathscr{C}}}+
C\,\big( {\widetilde{\mathscr{S}}}
+{\widetilde{\mathscr{T}}}
\big)&&{\mbox{in~$
B(x_0,R-\rho)\times [t_0-T+\delta,t_0]$,}}
\end{eqnarray*}
for some~$C>0$.
\end{cor}

Hence, considering the more convenient term in any common domain, we deduce
from Corollary~\ref{E67QRqwer3456788gfnh} that:

\begin{cor}
In the setting of Theorem~\ref{D2}, at any point in~$ Q_{R,T}$, we have that
\begin{equation}\label{C31234567xccgtt55d57374254777456dffA:0}\begin{split}
&w\le
C{\widetilde{\mathscr{C}}}+
\Big[\min\left\{\sigma_u^2+\tau_u^2,\,
\sigma_u^2+ C{\widetilde{\mathscr{T}}},\,
\tau_u^2+ C{\widetilde{\mathscr{S}}},\,
C({\widetilde{\mathscr{T}}}+{\widetilde{\mathscr{S}}})\right\}
\chi_{B(x_0,R-\rho)\times [t_0-T+\delta,t_0]}\\
&\qquad\qquad+\left(\sigma_u^2 +\min\left\{\tau_u^2,\,C{\widetilde{\mathscr{T}}}\right\}\right)\chi_{(B(x_0,R)\setminus B(x_0,R-\rho))\times [t_0-T+\delta,t_0]}\\
&\qquad\qquad+
\left(\tau_u^2 +\min\left\{\sigma_u^2,\,C{\widetilde{\mathscr{S}}}\right\}\right)
\chi_{B(x_0,R-\rho)\times [t_0-T,t_0-T+\delta]}\\
&\qquad\qquad+\left(\sigma_u^2 +\tau_u^2\right)\chi_{(B(x_0,R)\setminus B(x_0,R-\rho))\times [t_0-T,t_0-T+\delta]}
\Big],\end{split}
\end{equation}
for some~$C>0$.
\end{cor}

\begin{proof}[Completion of the proof of Theorem~\ref{D2}]
Recalling~\eqref{eaS}, we write that
$$ w=\frac{(F'(u))^{2}|\nabla u|^2}{u^2(\xi-G(u))^{2}}.
$$
This and~\eqref{C31234567xccgtt55d57374254777456dffA:0} give that
\begin{equation*}\begin{split}
&
\frac{(F'(u))^2|\nabla u|^2}{u^2(\xi-G(u))^{2}}\le
C{\widetilde{\mathscr{C}}}+
\Big[\min\left\{\sigma_u^2+\tau_u^2,\,
\sigma_u^2+ C{\widetilde{\mathscr{T}}},\,
\tau_u^2+ C{\widetilde{\mathscr{S}}},\,
C({\widetilde{\mathscr{T}}}+{\widetilde{\mathscr{S}}})\right\}
\chi_{B(x_0,R-\rho)\times [t_0-T+\delta,t_0]}\\
&\qquad\qquad+\left(\sigma_u^2 +\min\left\{\tau_u^2,\,C{\widetilde{\mathscr{T}}}\right\}\right)\chi_{(B(x_0,R)\setminus B(x_0,R-\rho))\times [t_0-T+\delta,t_0]}\\
&\qquad\qquad+
\left(\tau_u^2 +\min\left\{\sigma_u^2,\,C{\widetilde{\mathscr{S}}}\right\}\right)
\chi_{B(x_0,R-\rho)\times [t_0-T,t_0-T+\delta]}\\
&\qquad\qquad+\left(\sigma_u^2 +\tau_u^2\right)\chi_{(B(x_0,R)\setminus B(x_0,R-\rho))\times [t_0-T,t_0-T+\delta]}
\Big].\end{split}
\end{equation*}
Taking the square root and recalling~\eqref{TUTTITE-0}, we obtain
the desired result stated in Theorem~\ref{D2}.
\end{proof}

\section{Proof of Corollary \ref{noqwmefef004}}\label{8uCODroldd}

We now deduce a local estimate in~$Q_{R/2,T/2}$ as a special case of the global
one obtained in Theorem~\ref{D2}.

\begin{proof}[Proof of Corollary~\ref{noqwmefef004}]
By taking~$\delta:=T/2$ and~$\rho:=R/2$,
we reduce the quantities~${\mathscr{T}}$
and~${\mathscr{S}}$ in~\eqref{TUTTITE-0}
to
\begin{eqnarray*}
&&{\mathscr{T}}=\sqrt{\frac{2}{{T}}}\qquad{\mbox{ and }}\qquad
{\mathscr{S}}=\frac{4}{R}+\frac{\sqrt{2}\,\sqrt[4]{k_+}}{\sqrt{R}}.
\end{eqnarray*}
Consequently, we deduce from~\eqref{iotabeta} that
\begin{equation*}
\iota\le
C({{\mathscr{T}}}+{{\mathscr{S}}})\le C\left(
\frac{1}{R}+\frac1{\sqrt{T}}+\frac{\sqrt[4]{k_+}}{\sqrt{R}}
\right)
.\end{equation*}
Furthermore, by~\eqref{161}, we know that~$
{\mathscr{B}}_1={\mathscr{B}}_2=
{\mathscr{B}}_3=0$ in~$Q_{R/2,T/2}$
and therefore we deduce from
Theorem~\ref{D2} that, for each~$(x,t)\in Q_{R/2,T/2}$,
\begin{eqnarray*}
&&G'(u(x,t))\,|\nabla u(x,t)|\leq \Big(C
\mathscr{C}+
\iota\mathscr{I}(x,t)\Big)\,
\Big(\xi-G(u(x,t))\Big)\\&&\qquad\le
C
\left( \sqrt{\mu}+\sqrt[3]{\gamma}+
\frac{1}{R}+\frac1{\sqrt{T}}+\frac{\sqrt[4]{k_+}}{\sqrt{R}}
\right)\Big(\xi-G(u(x,t))\Big),
\end{eqnarray*}
as desired.
\end{proof}

\begin{appendix}

\section{A direct proof showing that
Corollary~\ref{noqwmefef004}
implies Corollary~\ref{99} (i.e., Corollary 9 in \cite{MR2392508})}\label{APP-90}

Goal of this appendix is to give a direct proof of
Corollary~\ref{99}
from the statement of Corollary~\ref{noqwmefef004}.
To this end, in the setting given by the statement of
Corollary~\ref{99}, we define~$F(s):=s^p$.
Let also
\begin{eqnarray}\label{Carthage}
\widetilde T:= M^{p-1}T
\qquad{\mbox{ and }}\qquad
\widetilde u(x,t):=\frac{u\big(x,M^{1-p}(t-t_0)+t_0\big)}{M}.\end{eqnarray}
We observe that if~$t\in [t_0-\widetilde T,t_0]$ then~$M^{1-p}(t-t_0)+t_0\in[t_0- T,t_0]$.
Consequently, if~$(x,t)\in Q_{R,\widetilde T}$,
\begin{eqnarray*}
\partial_t \widetilde u(x,t)=\frac{\partial_t u\big(x,M^{1-p}(t-t_0)+t_0\big)}{M^p}
=\frac{\Delta u^p\big(x,M^{1-p}(t-t_0)+t_0\big)}{M^p}=\Delta\widetilde u^p(x,t).
\end{eqnarray*}
Also, $0<\widetilde u\leq 1$, hence we can exploit
Corollary~\ref{noqwmefef004}, with~$T$ replaced by~$\widetilde T$, $u$ replaced by~$\widetilde u$
and~$M$ replaced by~$1$. Moreover, in~\eqref{DEFG} we pick
any~$s_0\ge2^{\frac1{1-p}}$
and we have that
$$ G(s)
=p\int_{s_0}^s h^{p-2}\,dh=\frac{p}{1-p}\left(\frac{1}{s_0^{1-p}}-\frac1{s^{1-p}}\right).$$
Then, choosing~$\eta:=\frac{p}{2(1-p)}$ and~$\xi:=0$, we have that, for all~$s\in(0,1]$,
\begin{equation}\label{767:0-scjkio-0}
\xi-G(s)=
\frac{p}{1-p}\left(\frac1{s^{1-p}}-\frac1{s_0^{1-p}}\right)
\end{equation}
and, as a byproduct,
\begin{equation}\label{767:0-scjkio-1}
\xi-G(s)
\ge
\frac{p}{1-p}\left(1-\frac1{s_0^{1-p}}\right)\ge\frac{p}{2(1-p)}
=\eta.\end{equation}
Furthermore,
$$ \frac{\sqrt{n}\,|F''(s)|\,s}{F'(s)}=
\sqrt{n} \,(1-p),$$
whence,
setting~$\kappa:=\sqrt{n}\left(p-1+\frac1{\sqrt{n}}\right)>0$,
\begin{equation}\label{767:0-scjkio-2}
1-\frac{\sqrt{n}\,|F''(s)|\,s}{F'(s)}=\kappa>0.
\end{equation}
Moreover, for all~$s\in(0,1]$,
\begin{eqnarray*}&& 2F'(s)-\frac{\sqrt{n}|F''(s)|s}{F'(s)}\big(\xi-G(s)\big)=
2p s^{p-1}-
\sqrt{n} \,(1-p)\,
\frac{p}{1-p}\left(\frac1{s^{1-p}}-\frac1{s_0^{1-p}}\right)\\&&\qquad
=\frac{2p}{ s^{1-p}}-p\sqrt{n}
\left(\frac1{s^{1-p}}-\frac1{s_0^{1-p}}\right)
=
\frac{p(2-\sqrt{n})}{ s^{1-p}}
+\frac{p\sqrt{n}}{s_0^{1-p}}\\&&\qquad\ge
p(2-\sqrt{4})
+\frac{p\sqrt{n}}{s_0^{1-p}}
\geq \frac{p\sqrt{n}}{s_0^{1-p}}\geq0
\end{eqnarray*}
{F}rom this, \eqref{767:0-scjkio-1} and~\eqref{767:0-scjkio-2},
we see that the conditions in~\eqref{kappapi}, \eqref{XIPOS}
and~\eqref{CONDFOD} are fulfilled.

We now check that~\eqref{SUPER} is also satisfied
(and, from the technical point of view, this step
is the one that makes assumption~\eqref{SUPER}
more convenient than~\eqref{CONDMA0}). To this end, we remark that,
in light of~\eqref{767:0-scjkio-0},
\begin{equation}\label{SIN8}
\xi-G(s)=\frac{p\,s^{p-1}}{1-p}-\frac{p\,s_0^{p-1}}{1-p},
\end{equation}
and therefore, for every~$s\in(0,1]$,
\begin{equation}\label{LAMCOSKLDA-lkjhgfMSDIG-PRE}
\frac{F'(s)}{\xi-G(s)}=\frac{(1-p)\,s^{p-1}}{s^{p-1}-s_0^{p-1}}=
\frac{ 1-p }{1-(s/s_0)^{1-p}},
\end{equation}
and consequently
\begin{equation}\label{LAMCOSKLDA-lkjhgfMSDIG}
\frac{F'(s)}{\xi-G(s)}\le\frac{ 1-p }{1-(1/s_0)^{1-p}}\le
2(1-p),
\end{equation}
and this shows that condition~\eqref{SUPER}
is fulfilled here with~$\Gamma:=2(1-p)$.

Therefore,
we can utilize Corollary~\ref{noqwmefef004} and conclude that,
if~$(x,t)\in Q_{R/2,\widetilde T/2}$,
\begin{equation}
\label{Nioksmdcwqur}\begin{split}&
\frac{G'(\widetilde u(x,t))\,|\nabla \widetilde u(x,t)|}{
\xi-G(\widetilde u(x,t))}=
\frac{|\nabla G(\widetilde u(x,t))|}{\xi-G(\widetilde u(x,t))}\\&\qquad
\le C\left(\frac1R+\frac1{\sqrt{\widetilde T}}+\sqrt{k}\right)=C\left(\frac1R+\frac{M^{\frac{1-p}2}}{\sqrt{T}}+\sqrt{k}\right).
\end{split}\end{equation}
Thus, recalling~\eqref{SIN8},
we find that
$$
\xi-G(s)=
\frac{s\,G'(s)}{1-p}-\frac{p\,s_0^{p-1}}{1-p},
$$
and then~\eqref{Nioksmdcwqur} gives that
\begin{eqnarray}\label{A.8}
(1-p)\,\frac{G'(\widetilde u(x,t))\,|\nabla \widetilde u(x,t)|}{
\widetilde u(x,t)\,G'(\widetilde u(x,t))-p\,s_0^{p-1}
}
\le C\left(\frac1R+\frac{M^{\frac{1-p}2}}{\sqrt{T}}+\sqrt{k}\right).
\end{eqnarray}
We can now send~$s_0\to+\infty$ and (up to renaming constants) conclude that,
for every~$(x,t)\in Q_{R/2,\widetilde T/2}$,
\begin{eqnarray}\label{A.9}
\frac{|\nabla \widetilde u(x,t)|}{
\widetilde u(x,t)
}
\le C\left(\frac1R+\frac{M^{\frac{1-p}2}}{\sqrt{T}}+\sqrt{k}\right),
\end{eqnarray}
which, scaling back the time variable, leads to the desired result in~\eqref{TIm94}.

\begin{rem} {\rm We stress that in this paper we are
not addressing the optimality of the range of powers~$p$
taken into account in the porous medium equation
dealt with in Corollary~\ref{99} (rather, the main goal of
this appendix was to show how to obtain some results in the literature,
such as Corollary~9 in~\cite{MR2392508},
as a byproduct of our main results). In a sense, we do not
expect the range of~$p$ presented here to be optimal
and it can be expected that broader intervals in~$p$ could
be addressed by combining our methods with those in~\cite{MR2853544}
(see in particular Remark~1.1 in~\cite{MR2853544}).
}\end{rem}

\section{The case of equation~\eqref{APPEPSE}}\label{APPEPS}

To emphasize the possible role
of the additional function~$H$ in
the evolution equation~\eqref{EQ}, we present here
a specific application (without aiming at the greatest possible generality):

\begin{cor}\label{NVDGNUIJGNFGNTGBPROSMGN}
Let~$\mathscr{M}$
be a complete Riemannian manifold of dimension~$n$
with~$\mathrm{Ric}(\mathscr{M})\geq-k$, for some~$k\ge0$.
Let~$ u=u(x,t)$ be a positive solution of~\eqref{APPEPSE} in~$Q_{R,T}$,
with~$ 0<m\leq u\le M$ for some~$M>0$, $m>0$,
$\varepsilon>0$, $a_0\in(0,1)$, $q>0$,
$p$ as in~\eqref{RANGEP} and~$a\in C^1(\R)$ such that~$a(t)
\in [a_0,a_0^{-1}]$ for all~$t\in\R$.

Let also
\begin{equation}\label{BAK89DAPPLDBLASALBODL} {\mathscr{F}}(R,T):=\sup_{(x,t)\in Q_{R,T}}
\frac{|\nabla u(x,t)|}{ u(x,t) }\qquad
{\mbox{ and }}\qquad{\mathscr{H}}(R,T):=
\sup_{(x,t)\in Q_{R,T}}\frac{|D^2 u(x,t)|}{\big(u(x,t)\big)^{3-p-q}}.\end{equation}
Then,
there exists~$C>0$, depending only on~$n$, $a_0$, $q$, and~$p$~such that
\begin{equation}\label{NSDVENmOIDKN45367SDHFI}
{\mathscr{F}}\left(\frac{R}2,\frac{T}2\right)\le
C\,
\left(
\frac{ \sqrt{k} M^{\frac{1-p}2}}{ m^{\frac{1-p}2}}
+\sqrt[3]{\varepsilon}\; M^{\frac{2-2p}3}\,
\big({\mathscr{F}}(R,T)\big)^{\frac{q-1}3}\,\sqrt[3]{{\mathscr{H}}(R,T)
}+
\frac{1}{R}+\frac{M^{\frac{1-p}2}}{\sqrt{T}}+\frac{\sqrt[4]{k_+}}{\sqrt{R}}
\right).\end{equation}
\end{cor}

\begin{proof} 
We let~$\widetilde T$ and~$\widetilde u$ as in~\eqref{Carthage}
and~$\widetilde a(t):=a(M^{1-p}(t-t_0)+t_0)$.
Let also
$$ F(s):=s^p\quad{\mbox{ and }}\quad H(\omega):=\varepsilon M^{q-p} \,|\omega|^q.$$
In this way, we have that~$\frac{m}M\le \widetilde{u}\le1$ and, by~\eqref{APPEPSE},
\begin{equation}\label{APPEPSE-SVALESDF}
\begin{split}
\partial_t \widetilde u(x,t)\,&=\,\frac{\partial_tu\big(x,M^{1-p}(t-t_0)+t_0\big)}{M^p}
\\&=\,\frac{
a(M^{1-p}(t-t_0)+t_0)\;\Delta u^p(x,M^{1-p}(t-t_0)+t_0)+\varepsilon\,|\nabla u(x,M^{1-p}(t-t_0)+t_0)|^{q}
}{M^p}\\
&=\,\widetilde a(t) \Delta \widetilde{u}^p(x,t)
+\frac{\varepsilon\,|\nabla\widetilde{ u}(x,t)|^{q}}{M^{p-q}}\\&
=\,\widetilde a(t)
\Delta(F(\widetilde{u}))+H(\nabla\widetilde{ u}),
\end{split}\end{equation}
and thus
we see that~\eqref{APPEPSE-SVALESDF} is a special case of~\eqref{EQ}.

As already remarked in Appendix~\ref{APP-90},
we have that
conditions~\eqref{FPRIM},
\eqref{kappapi}, \eqref{XIPOS}, \eqref{SUPER}
and~\eqref{CONDFOD}
are satisfied by~$F$ (with suitable~$s_0$, $\xi$, $\eta$, $\Gamma$
and~$\kappa$ depending only on~$n$ and~$p$). 

Moreover,
we have that~$\widetilde{a}\in [a_0,a_0^{-1}]$, whence
we deduce from~\eqref{DEFmu}, \eqref{MUDUE}
and~\eqref{LAMCOSKLDA-lkjhgfMSDIG} that
\begin{eqnarray*}
\mu&=&\mu_1\\
&\le&\sup_{(x,t)\in Q_{R,\widetilde{T}}}\Bigg(a^{-1}_0 kp \big(\widetilde u(x,t)\big)^{p-1}
+\frac{\varepsilon\,(p-1)\,M^{q-p}\,|\nabla \widetilde u(x,t)|^q\,}{\widetilde u(x,t)
}
\\&&\qquad\qquad
-\frac{\varepsilon\,M^{q-p}\,|\nabla \widetilde u(x,t)|^q}{\widetilde u(x,t)
}+
\frac{2\varepsilon\,(1-p)\,M^{q-p}\,|\nabla \widetilde u(x,t)|^q}{
\widetilde u(x,t)}\Bigg)_+\\&=&
\sup_{(x,t)\in Q_{R,\widetilde{T}}}\Bigg(a^{-1}_0 kp \big(\widetilde u(x,t)\big)^{p-1}
-\frac{p\varepsilon\,M^{q-p}\,|\nabla \widetilde u(x,t)|^q}{
\widetilde u(x,t)}\Bigg)_+\\&\le&
\sup_{(x,t)\in Q_{R,\widetilde{T}}}\Bigg(a^{-1}_0 kp \big(\widetilde u(x,t)\big)^{p-1}
\Bigg)_+
\\&\le&
\frac{ kp M^{1-p}}{a_0 m^{1-p}}.
\end{eqnarray*}
Also, owing to~\eqref{DEFgamma}, \eqref{GAMMA2}
and~\eqref{GAMMA32},
\begin{eqnarray*}
\gamma&=&\gamma_3 \\&\le&
\sup_{(x,t)\in Q_{R,\widetilde{T}}}
pq\varepsilon\,M^{q-p}\big(\widetilde u(x,t)\big)^{p-2}
|\nabla \widetilde u(x,t)|^{q-1}
\,|D^2\widetilde u(x,t)|\\&
=&\sup_{(x,t)\in Q_{R,{T}}}
pq\varepsilon\,M^{2-2p}\big( u(x,t)\big)^{p-2}
|\nabla u(x,t)|^{q-1}
\,|D^2 u(x,t)|
\\&\le&
{pq\varepsilon M^{2-2p}}\,
\big({\mathscr{F}}(R,T)\big)^{q-1}\,{\mathscr{H}}(R,T).
\end{eqnarray*}
From~\eqref{767:0-scjkio-0} (and sending~$s_0\to+\infty$), utilizing
Corollary~\ref{noqwmefef004}, and renaming~$C$ line after line,
for every~$(x,t)\in Q_{R/2,T/2}$,
\begin{eqnarray*}
\sup_{Q_{R/2,T/2}}\frac{\,|\nabla u|}{u}&=&
\sup_{Q_{R/2,\widetilde{T}/2}}\frac{\,|\nabla\widetilde{ u}|}{\widetilde{u}}\\&
\le& C
\left( \sqrt{\mu}+\sqrt[3]{\gamma}+
\frac{1}{R}+\frac1{\sqrt{\widetilde{T}}}+\frac{\sqrt[4]{k_+}}{\sqrt{R}}
\right)
\\&\le&C
\left(
\frac{ \sqrt{k} M^{\frac{1-p}2}}{\sqrt{a_0} m^{\frac{1-p}2}}
+\sqrt[3]{\varepsilon}\; M^{\frac{2-2p}3}\,
\big({\mathscr{F}}(R,T)\big)^{\frac{q-1}3}\,\sqrt[3]{{\mathscr{H}}(R,T)
}+
\frac{1}{R}+\frac{M^{\frac{1-p}2}}{\sqrt{T}}+\frac{\sqrt[4]{k_+}}{\sqrt{R}}
\right),
\end{eqnarray*}
that gives the desired result.
\end{proof}

We remark that estimate~\eqref{NSDVENmOIDKN45367SDHFI}
is certainly nonstandard in bounding~$
{\mathscr{F}}\left(\frac{R}2,\frac{T}2\right)$
with the obviously larger term~$
{\mathscr{F}}(R,T)$ and with the term~$
{\mathscr{H}}(R,T)$
that involves higher derivatives: nevertheless,
estimate~\eqref{NSDVENmOIDKN45367SDHFI} is nontrivial,
since these larger or higher order objects
occur with a lower exponent when~$q<4$
and are modulated by the (possibly small)
structural parameter~$\varepsilon$.

In this spirit, we point out a quantitative result
on the oscillations
of ancient solutions
which easily follows from Corollary~\ref{NVDGNUIJGNFGNTGBPROSMGN}:

\begin{cor}
Assume that
\begin{equation}\label{NONNERI-BIS-CO}\begin{split}&
{\mbox{${\mathscr{M}}$ is
a complete Riemannian manifold of dimension~$n$}}\\&{\mbox{with nonnegative Ricci curvature.}}\end{split}\end{equation}
Let~$u$ be a positive, bounded and smooth solution
of the evolution equation
\begin{equation*}
\partial_t u=a(t)\Delta u^p+\varepsilon\,|\nabla u|^{q}
\end{equation*}
in~${\mathscr{M}}\times(-\infty,0]$, for some~$a\in C^1((-\infty,0])$
which is positive, bounded and bounded away from zero,
some~$\varepsilon>0$, $
p\in\left(1-\frac1{\sqrt{n}},1\right]$
and~$q\in(0,4)$.

Assume that
\begin{equation}\label{SUPNUOSPESA} {\mathscr{F}}_\star:=\sup_{{\mathscr{M}}\times(-\infty,0]}
\frac{|\nabla u |}{ u }<+\infty\qquad
{\mbox{ and }}\qquad {\mathscr{H}}_\star:=
\sup_{{\mathscr{M}}\times(-\infty,0]}\frac{|D^2 u|}{u^{3-p-q}}<+\infty.
\end{equation}

Then, there exists~$C>0$ depending only on~$n$, $a$, $q$
and~$p$ such that
$$
{\mathscr{F}}_\star\le
C\, \varepsilon^{\frac1{4-q}}\; M^{\frac{2-2p}{4-q}}\,
{{\mathscr{H}}_\star^{\frac1{4-q}}
},
$$
where~$M$ is the supremum of~$u$. 
\end{cor}

\begin{proof} We let~$
{\mathscr{F}}(R,T)$ and~$
{\mathscr{H}}(R,T)$ be as in~\eqref{BAK89DAPPLDBLASALBODL}.
With this notation, and recalling~\eqref{NONNERI-BIS-CO},
we can exploit~\eqref{NSDVENmOIDKN45367SDHFI}
with~$t_0:=0$ and any given~$x_0\in{\mathscr{M}}$,
thereby finding that, for every~$R>0$, $T>0$, $x\in
B(x_0,R/2)$ and~$t\in [-T/2,0]$,
\begin{equation*}
{\mathscr{F}}\left(\frac{R}2,\frac{T}2\right)\le
C\,
\left(\sqrt[3]{\varepsilon}\; M^{\frac{2-2p}3}\,
\big({\mathscr{F}}(R,T)\big)^{\frac{q-1}3}\,\sqrt[3]{{\mathscr{H}}(R,T)
}+
\frac{1}{R}+\frac{M^{\frac{1-p}2}}{\sqrt{T}}
\right).\end{equation*}
Consequently,
we can send~$R\to+\infty$ and~$T\to+\infty$
and obtain that
\begin{equation*}
{\mathscr{F}}_\star\le
C\,\sqrt[3]{\varepsilon}\; M^{\frac{2-2p}3}\,
{\mathscr{F}}_\star^{\frac{q-1}3}\,\sqrt[3]{{\mathscr{H}}_\star
},\end{equation*}
from which the desired result follows.
\end{proof}

\section{A Liouville-type result}\label{DFDF}

In this appendix we point out that
suitable
classification results for ancient solutions under appropriate
pointwise bounds follows directly from uniform interior estimates.
We do not address the most general statement here,
but just provide the following one as a simple example:

\begin{cor}
Assume that
\begin{equation}\label{NONNERI}\begin{split}&
{\mbox{${\mathscr{M}}$ is
a complete Riemannian manifold of dimension~$n$}}\\&{\mbox{with nonnegative Ricci curvature.}}\end{split}\end{equation}
Let~$u$ be a positive and smooth solution
of the evolution equation
\begin{equation}\label{MS DMDOMFNODKMSKR}
u_{t}=a(t)\,\Delta(F(u))\end{equation}
in~${\mathscr{M}}\times(-\infty,0]$, for some~$a\in C^1((-\infty,0])$
which is positive, bounded and bounded away from zero,
and for some~$F\in C^{2}(0,+\infty)$ satisfying~\eqref{FPRIM},
\eqref{kappapi}, \eqref{XIPOS}, \eqref{SUPER}
and~\eqref{CONDFOD}.

Assume that
\begin{equation}\label{SUPPE} \sup_{{x\in B_R}\atop{t\in[-T,0]}}u(x,t)=o(R)+o(\sqrt{T})\end{equation}
near infinity. Then, $u$ is constant.
\end{cor}

\begin{proof} We exploit Corollary~\ref{noqwmefef004}.
For this, we observe that we can take~$k:=0$, owing to~\eqref{NONNERI}.
Furthermore, comparing~\eqref{MS DMDOMFNODKMSKR}
with~\eqref{DEFmu}, \eqref{DEFgamma}, \eqref{MUDUE},
\eqref{DEFMU22}, \eqref{GAMMA2}, \eqref{GAMMA32}
and~\eqref{DEFMU33}, we see that~$\mu=\gamma=0$.

For this reason, for every~$R>0$, $T>0$, $x\in B_{R/2}$ and~$t\in[-T/2,0]$
the use of Corollary~\ref{noqwmefef004} leads to
\begin{equation*}
G'(u(x,t))\,|\nabla u(x,t)|\leq C
\left(
\frac{1}{R}+\frac1{\sqrt{T}}
\right)\Big(\xi-G(u(x,t))\Big),\end{equation*}
for a suitable structural constant~$C$.

This and~\eqref{DEFG} give that
\begin{equation*}
|\nabla u(x,t)|\leq C
\left(
\frac{1}{R}+\frac1{\sqrt{T}}
\right)\,\frac{\big(\xi-G(u(x,t))\big) \,u(x,t)}{F'(u(x,t))}.\end{equation*}
Hence, in light of~\eqref{CONDFOD},
\begin{equation*}
|\nabla u(x,t)|\leq C
\left(
\frac{1}{R}+\frac1{\sqrt{T}}
\right)\,\frac{F'(u(x,t))}{|F''(u(x,t))|},\end{equation*}
and thus, by~\eqref{kappapi},
\begin{equation*}
|\nabla u(x,t)|\leq C
\left(
\frac{1}{R}+\frac1{\sqrt{T}}
\right)\, u(x,t),\end{equation*}
up to renaming~$C$ line after line.

Therefore, utilizing~\eqref{SUPPE}, for every~$x\in B_{R/2}$
and~$t\in [-R^2/2,0]$,
\begin{equation*}
|\nabla u(x,t)|\leq
\frac{C}{R}\;
\sup_{B_R\times[-R^2,0]} u=\frac{o(R)}R.\end{equation*}
Sending now~$R\to+\infty$ we obtain the desired result.
\end{proof}

With respect to~\eqref{SUPPE},
we do not address here the problem of the optimal
growth at infinity needed to obtain nontrivial solutions.
As a matter of fact, we do not expect condition~\eqref{SUPPE}
to be optimal (indeed, at least when~$a$ is independent of~$t$
and~$F$ is a suitable power, milder growth assumptions
lead to suitable classification results,
see e.g. Theorems~1.3 and~1.5 in~\cite{MR2853544}).

\end{appendix}

\vskip4mm
\par\noindent

\section*{Acknowledgments}
\vskip2mm
\noindent
Serena Dipierro and Enrico Valdinoci are members of INdAM and AustMS.
Serena Dipierro has been supported by
the Australian Research Council DECRA DE180100957
``PDEs, free boundaries and applications''.
Enrico Valdinoci has been supported by
the Australian Laureate Fellowship
FL190100081
``Minimal surfaces, free boundaries and partial differential equations''.


\begin{biblist}\begin{bibdiv}

\bib{MR524760}{article}{
   author={Aronson, Donald G.},
   author={B\'{e}nilan, Philippe},
   title={R\'{e}gularit\'{e} des solutions de l'\'{e}quation des milieux poreux dans
   ${\bf R}^{N}$},
   language={French, with English summary},
   journal={C. R. Acad. Sci. Paris S\'{e}r. A-B},
   volume={288},
   date={1979},
   number={2},
   pages={A103--A105},
   issn={0151-0509},
   review={\MR{524760}},
}

\bib{MR1219716}{article}{
   author={Auchmuty, Giles},
   author={Bao, David},
   title={Harnack-type inequalities for evolution equations},
   journal={Proc. Amer. Math. Soc.},
   volume={122},
   date={1994},
   number={1},
   pages={117--129},
   issn={0002-9939},
   review={\MR{1219716}},
   doi={10.2307/2160850},
}

\bib{MR1720778}{article}{
   author={Benachour, Said},
   author={Lauren\c{c}ot, Philippe},
   title={Global solutions to viscous Hamilton-Jacobi equations with
   irregular initial data},
   journal={Comm. Partial Differential Equations},
   volume={24},
   date={1999},
   number={11-12},
   pages={1999--2021},
   issn={0360-5302},
   review={\MR{1720778}},
   doi={10.1080/03605309908821492},
}

\bib{MR3735744}{article}{
   author={Bianchi, Davide},
   author={Setti, Alberto G.},
   title={Laplacian cut-offs, porous and fast diffusion on manifolds and
   other applications},
   journal={Calc. Var. Partial Differential Equations},
   volume={57},
   date={2018},
   number={1},
   pages={Paper No. 4, 33},
   issn={0944-2669},
   review={\MR{3735744}},
   doi={10.1007/s00526-017-1267-9},
}

\bib{MR2383484}{article}{
   author={Bonforte, Matteo},
   author={Grillo, Gabriele},
   author={Vazquez, Juan Luis},
   title={Fast diffusion flow on manifolds of nonpositive curvature},
   journal={J. Evol. Equ.},
   volume={8},
   date={2008},
   number={1},
   pages={99--128},
   issn={1424-3199},
   review={\MR{2383484}},
   doi={10.1007/s00028-007-0345-4},
}
	
\bib{MR3621772}{article}{
   author={Castorina, Daniele},
   author={Mantegazza, Carlo},
   title={Ancient solutions of semilinear heat equations on Riemannian
   manifolds},
   journal={Atti Accad. Naz. Lincei Rend. Lincei Mat. Appl.},
   volume={28},
   date={2017},
   number={1},
   pages={85--101},
   issn={1120-6330},
   review={\MR{3621772}},
   doi={10.4171/RLM/753},
}

\bib{12345}{article}{
   author={Castorina, Daniele},
   author={Mantegazza, Carlo},
   title={Ancient solutions of superlinear heat equations on
Riemannian manifolds},
   journal={Commun. Contemp. Math.},
   date={to appear},
}

\bib{2019arXiv190304569C}{article}{
  author={Cavaterra, Cecilia},
   author={Dipierro, Serena},
   author={Farina, Alberto},
   author={Gao, Zu},
   author={Valdinoci, Enrico},
   title={Pointwise gradient bounds for entire solutions of elliptic
   equations with non-standard growth conditions and general nonlinearities},
   journal={J. Differential Equations},
   volume={270},
   date={2021},
   pages={435--475},
   issn={0022-0396},
   review={\MR{4150380}},
   doi={10.1016/j.jde.2020.08.007},
}

\bib{ZU2}{article}{
       author ={Cavaterra, Cecilia},
     author ={Dipierro, Serena},
     author = {Gao, Zu},
     author ={Valdinoci, Enrico},
        title = {Global gradient estimates for a general type of nonlinear parabolic equations},
      journal = {arXiv e-prints},
date={2020},
eid = {arXiv:2006.00263},
        pages = {arXiv:2006.00263},
archivePrefix = {arXiv},
       eprint = {2006.00263},
 primaryClass = {math.AP},
       adsurl = {https://ui.adsabs.harvard.edu/abs/2020arXiv200600263C},
      adsnote = {Provided by the SAO/NASA Astrophysics Data System}
}

\bib{MR897437}{article}{
   author={Dal Passo, Roberta},
   author={Luckhaus, Stephan},
   title={A degenerate diffusion problem not in divergence form},
   journal={J. Differential Equations},
   volume={69},
   date={1987},
   number={1},
   pages={1--14},
   issn={0022-0396},
   review={\MR{897437}},
   doi={10.1016/0022-0396(87)90099-4},
}

\bib{MR2856180}{book}{
   author={Fourier, Jean Baptiste Joseph},
   title={Th\'{e}orie analytique de la chaleur},
   language={French},
   series={Cambridge Library Collection},
   note={Reprint of the 1822 original;
   Previously published by \'{E}ditions Jacques Gabay, Paris, 1988 [MR1414430]},
   publisher={Cambridge University Press, Cambridge},
   date={2009},
   pages={ii+xxii+643},
   isbn={978-1-108-00180-9},
   review={\MR{2856180}},
   doi={10.1017/CBO9780511693229},
}

\bib{MR0181836}{book}{
   author={Friedman, Avner},
   title={Partial differential equations of parabolic type},
   publisher={Prentice-Hall, Inc., Englewood Cliffs, N.J.},
   date={1964},
   pages={xiv+347},
   review={\MR{0181836}},
}

\bib{MR3427985}{article}{
   author={Grillo, Gabriele},
   author={Muratori, Matteo},
   title={Smoothing effects for the porous medium equation on
   Cartan-Hadamard manifolds},
   journal={Nonlinear Anal.},
   volume={131},
   date={2016},
   pages={346--362},
   issn={0362-546X},
   review={\MR{3427985}},
   doi={10.1016/j.na.2015.07.029},
}

\bib{MR3658720}{article}{
   author={Grillo, Gabriele},
   author={Muratori, Matteo},
   author={V\'{a}zquez, Juan Luis},
   title={The porous medium equation on Riemannian manifolds with negative
   curvature. The large-time behaviour},
   journal={Adv. Math.},
   volume={314},
   date={2017},
   pages={328--377},
   issn={0001-8708},
   review={\MR{3658720}},
   doi={10.1016/j.aim.2017.04.023},
}

\bib{MR1230276}{article}{
   author={Hamilton, Richard S.},
   title={A matrix Harnack estimate for the heat equation},
   journal={Comm. Anal. Geom.},
   volume={1},
   date={1993},
   number={1},
   pages={113--126},
   issn={1019-8385},
   review={\MR{1230276}},
   doi={10.4310/CAG.1993.v1.n1.a6},
}

\bib{MR797051}{article}{
   author={Herrero, Miguel A.},
   author={Pierre, Michel},
   title={The Cauchy problem for $u_t=\Delta u^m$ when $0<m<1$},
   journal={Trans. Amer. Math. Soc.},
   volume={291},
   date={1985},
   number={1},
   pages={145--158},
   issn={0002-9947},
   review={\MR{797051}},
   doi={10.2307/1999900},
}

\bib{MR1618694}{article}{
   author={Hsu, Elton P.},
   title={Estimates of derivatives of the heat kernel on a compact
   Riemannian manifold},
   journal={Proc. Amer. Math. Soc.},
   volume={127},
   date={1999},
   number={12},
   pages={3739--3744},
   issn={0002-9939},
   review={\MR{1618694}},
   doi={10.1090/S0002-9939-99-04967-9},
}

\bib{MR834612}{article}{
   author={Li, Peter},
   author={Yau, Shing-Tung},
   title={On the parabolic kernel of the Schr\"{o}dinger operator},
   journal={Acta Math.},
   volume={156},
   date={1986},
   number={3-4},
   pages={153--201},
   issn={0001-5962},
   review={\MR{834612}},
   doi={10.1007/BF02399203},
}

\bib{MR2487898}{article}{
   author={Lu, Peng},
   author={Ni, Lei},
   author={V\'{a}zquez, Juan-Luis},
   author={Villani, C\'{e}dric},
   title={Local Aronson-B\'{e}nilan estimates and entropy formulae for porous
   medium and fast diffusion equations on manifolds},
   language={English, with English and French summaries},
   journal={J. Math. Pures Appl. (9)},
   volume={91},
   date={2009},
   number={1},
   pages={1--19},
   issn={0021-7824},
   review={\MR{2487898}},
   doi={10.1016/j.matpur.2008.09.001},
}

\bib{MR2264255}{article}{
   author={Ma, Li},
   title={Gradient estimates for a simple elliptic equation on complete
   non-compact Riemannian manifolds},
   journal={J. Funct. Anal.},
   volume={241},
   date={2006},
   number={1},
   pages={374--382},
   issn={0022-1236},
   review={\MR{2264255}},
   doi={10.1016/j.jfa.2006.06.006},
}

\bib{MR1742035}{article}{
   author={Ma, Li},
   author={An, Yinglian},
   title={The maximum principle and the Yamabe flow},
   conference={
      title={Partial differential equations and their applications},
      address={Wuhan},
      date={1999},
   },
   book={
      publisher={World Sci. Publ., River Edge, NJ},
   },
   date={1999},
   pages={211--224},
   review={\MR{1742035}},
}

\bib{MR2392508}{article}{
   author={Ma, Li},
   author={Zhao, Lin},
   author={Song, Xianfa},
   title={Gradient estimate for the degenerate parabolic equation
   $u_t=\Delta F(u)+H(u)$ on manifolds},
   journal={J. Differential Equations},
   volume={244},
   date={2008},
   number={5},
   pages={1157--1177},
   issn={0022-0396},
   review={\MR{2392508}},
   doi={10.1016/j.jde.2007.08.014},
}

\bib{MR1431005}{article}{
   author={Malliavin, Paul},
   author={Stroock, Daniel W.},
   title={Short time behavior of the heat kernel and its logarithmic
   derivatives},
   journal={J. Differential Geom.},
   volume={44},
   date={1996},
   number={3},
   pages={550--570},
   issn={0022-040X},
   review={\MR{1431005}},
}
	
\bib{MR2285258}{article}{
   author={Souplet, Philippe},
   author={Zhang, Qi S.},
   title={Sharp gradient estimate and Yau's Liouville theorem for the heat
   equation on noncompact manifolds},
   journal={Bull. London Math. Soc.},
   volume={38},
   date={2006},
   number={6},
   pages={1045--1053},
   issn={0024-6093},
   review={\MR{2285258}},
   doi={10.1112/S0024609306018947},
}

\bib{MR1664888}{article}{
   author={Stroock, Daniel W.},
   author={Turetsky, James},
   title={Upper bounds on derivatives of the logarithm of the heat kernel},
   journal={Comm. Anal. Geom.},
   volume={6},
   date={1998},
   number={4},
   pages={669--685},
   issn={1019-8385},
   review={\MR{1664888}},
   doi={10.4310/CAG.1998.v6.n4.a2},
}

\bib{MR859613}{article}{
   author={Ughi, Maura},
   title={A degenerate parabolic equation modelling the spread of an
   epidemic},
   journal={Ann. Mat. Pura Appl. (4)},
   volume={143},
   date={1986},
   pages={385--400},
   issn={0003-4622},
   review={\MR{859613}},
   doi={10.1007/BF01769226},
}

\bib{MR2286292}{book}{
   author={V\'{a}zquez, Juan Luis},
   title={The porous medium equation},
   series={Oxford Mathematical Monographs},
   note={Mathematical theory},
   publisher={The Clarendon Press, Oxford University Press, Oxford},
   date={2007},
   pages={xxii+624},
   isbn={978-0-19-856903-9},
   isbn={0-19-856903-3},
   review={\MR{2286292}},
}

\bib{MR3633806}{article}{
   author={Wang, Wen},
   title={Harnack differential inequalities for the parabolic equation
   $u_t=\scr{L}F(u)$ on Riemannian manifolds and applications},
   journal={Acta Math. Sin. (Engl. Ser.)},
   volume={33},
   date={2017},
   number={5},
   pages={620--634},
   issn={1439-8516},
   review={\MR{3633806}},
   doi={10.1007/s10114-016-6260-2},
}

\bib{MR2853544}{article}{
   author={Xu, Xiangjin},
   title={Gradient estimates for $u_t=\Delta F(u)$ on manifolds and some
   Liouville-type theorems},
   journal={J. Differential Equations},
   volume={252},
   date={2012},
   number={2},
   pages={1403--1420},
   issn={0022-0396},
   review={\MR{2853544}},
   doi={10.1016/j.jde.2011.08.004},
}

\bib{MR2425752}{article}{
   author={Yang, Yunyan},
   title={Gradient estimates for a nonlinear parabolic equation on
   Riemannian manifolds},
   journal={Proc. Amer. Math. Soc.},
   volume={136},
   date={2008},
   number={11},
   pages={4095--4102},
   issn={0002-9939},
   review={\MR{2425752}},
   doi={10.1090/S0002-9939-08-09398-2},
}

\bib{MR2763753}{article}{
   author={Zhu, Xiaobao},
   title={Hamilton's gradient estimates and Liouville theorems for fast
   diffusion equations on noncompact Riemannian manifolds},
   journal={Proc. Amer. Math. Soc.},
   volume={139},
   date={2011},
   number={5},
   pages={1637--1644},
   issn={0002-9939},
   review={\MR{2763753}},
   doi={10.1090/S0002-9939-2010-10824-9},
}

\bib{MR3023250}{article}{
   author={Zhu, Xiaobao},
   title={Hamilton's gradient estimates and Liouville theorems for porous
   medium equations on noncompact Riemannian manifolds},
   journal={J. Math. Anal. Appl.},
   volume={402},
   date={2013},
   number={1},
   pages={201--206},
   issn={0022-247X},
   review={\MR{3023250}},
   doi={10.1016/j.jmaa.2013.01.018},
}

\end{bibdiv}\end{biblist}
\end{document}